\documentclass[11pt,reqno,tbtags]{amsart}
\usepackage{amssymb}
\pdfoutput=1
\usepackage[numeric,initials,nobysame]{amsrefs}
\usepackage{amsmath,amssymb,amsfonts,dsfont}
\usepackage{enumerate,graphicx}

\numberwithin{equation}{section}

\newtheorem{maintheorem}{Theorem}
\newtheorem{theorem}{Theorem}[section]
\newtheorem*{theorem*}{Theorem}
\newtheorem{conjecture}[theorem]{Conjecture}
\newtheorem{lemma}[theorem]{Lemma}

\newtheorem{proposition}[theorem]{Proposition}
\newtheorem{corollary}[theorem]{Corollary}

\theoremstyle{definition}{

\newtheorem*{definition*}{Definition}

}

\theoremstyle{remark}{

\newtheorem*{remark*}{Remark}

}

\newenvironment{enumeratei}{\begin{enumerate}[\upshape (i)]}
                           {\end{enumerate}}

\newcommand{\R}{\mathbb R}

\newcommand{\E}{\mathbb{E}}
\renewcommand{\P}{\mathbb{P}}
\DeclareMathOperator{\var}{Var}

\DeclareMathOperator{\sign}{sign}

\renewcommand{\epsilon}{\varepsilon}
\newcommand{\tS}{\tilde{S}}
\newcommand{\dist}{{\rm dist}}

\newcommand{\given}{\, \big| \,}

\newcommand{\tX}{\tilde{X}}

\newcommand{\taumag}{\tau_{{\rm mag}}}

\newcommand{\one}{\boldsymbol{1}}

\newcommand{\deq}{\stackrel{\scriptscriptstyle\triangle}{=}}

\newcommand{\tmix}{t_{{\rm mix}}}

\newcommand{\F}{\mathcal{F}}

\newcommand{\temp}{\delta}
\newcommand{\cS}{\mathcal{S}}
\newcommand{\cZ}{\mathcal{Z}}
\newcommand{\cX}{\mathcal{X}}
\newcommand{\cP}{\mathcal{P}}
\newcommand{\gap}{\text{\tt{gap}}}
\newcommand{\magspace}{\Psi}
\DeclareMathOperator{\lip}{lip}

\begin{document}
\title[Censored dynamics for the mean field Ising Model]{Censored Glauber Dynamics for \\ the mean field Ising Model}
%\date{\today}
\date{}

\author{Jian Ding, \thinspace Eyal Lubetzky and Yuval Peres}

%\author{Jian Ding}
\address{Jian Ding\hfill\break
Department of Statistics\\
UC Berkeley\\
Berkeley, CA 94720, USA.}
\email{jding@stat.berkeley.edu}
\urladdr{}

%\author{Eyal Lubetzky}
\address{Eyal Lubetzky\hfill\break
Microsoft Research\\
One Microsoft Way\\
Redmond, WA 98052-6399, USA.}
\email{eyal@microsoft.com}
\urladdr{}

%\author{Yuval Peres}
\address{Yuval Peres\hfill\break
Microsoft Research\\
One Microsoft Way\\
Redmond, WA 98052-6399, USA.} \email{peres@microsoft.com}
\urladdr{}

\begin{abstract}
We study Glauber dynamics for the Ising model on the complete graph on $n$ vertices,
known as the Curie-Weiss Model. It is well known that at high
temperature ($\beta < 1$) the mixing time is $\Theta(n\log n)$,
whereas at low temperature ($\beta > 1$) it is $\exp(\Theta(n))$. Recently, Levin,
Luczak and Peres considered a \emph{censored} version of this dynamics, which is restricted to
non-negative magnetization. They proved that for fixed $\beta > 1$, the mixing-time of this model is
$\Theta(n\log n)$, analogous to the high-temperature regime
of the original dynamics. Furthermore, they showed \emph{cutoff} for the original dynamics
for fixed $\beta<1$. The question whether the censored dynamics also exhibits cutoff remained
unsettled.

In a companion paper, we extended the results of Levin et al. into a complete characterization of the mixing-time for the Currie-Weiss model. Namely, we found a scaling window of order $1/\sqrt{n}$ around the critical temperature $\beta_c=1$, beyond which there is cutoff at high temperature. However, determining the behavior of the censored dynamics outside this critical window seemed significantly more challenging.

In this work we answer the above question in the affirmative, and establish the cutoff point and its
window for the censored dynamics beyond the critical window, thus completing its analogy to the original
dynamics at high temperature. Namely, if $\beta = 1 + \temp$ for some $\temp > 0$ with $\temp^2 n \to \infty$, then the mixing-time has order $(n/\temp)\log(\temp^2 n)$. The cutoff constant is $\left(1/2+[2(\zeta^2 \beta/\temp - 1)]^{-1}\right)$, where $\zeta$ is the unique positive root of $g(x)=\tanh(\beta x)-x$, and the cutoff window
has order $n/\temp$.
\end{abstract}

\maketitle

%start
\section{Introduction}\label{sec:intro}

The \emph{Ising Model} on a finite graph $G=(V,E)$ with parameter
$\beta \geq 0$ and no external magnetic field is defined as
follows. Its set of possible \emph{configurations} is $\Omega =
\{1,-1\}^V$, where each configuration $\sigma\in\Omega$ assigns
positive or negatives \emph{spins} to the vertices of the graph.
The probability that the system is at the configuration $\sigma$ is
given by the \emph{Gibbs distribution}
$$\mu_G(\sigma) = \frac{1}{Z(\beta)} \exp\Big(\beta \sum_{xy\in E}\sigma(x)\sigma(y)\Big)~,$$
where $Z(\beta)$ (the partition function) serves as a normalizing
constant. The parameter $\beta$ represents the inverse
temperature: the higher $\beta$ is (the lower the temperature is),
the more $\mu_G$ favors configurations where neighboring spins are
aligned. At the extreme case $\beta = 0$ (infinite temperature),
the spins are completely independent and $\mu_G$ is uniform over
$\Omega$.

The \emph{Curie-Weiss} model corresponds to the case where the
underlying geometry is the complete graph on $n$ vertices. The
study of this model (see, e.g.,
\cite{Ellis},\cite{EN},\cite{ENR},\cite{LLP}) is motivated by the
fact that its behavior approximates that of the Ising model on
high-dimensional tori. It is convenient in this case to rescale
the parameter $\beta$, so that the stationary measure $\mu_n$
satisfies
\begin{equation}\label{eq-mu(sigma)}
\mu_n(\sigma) \propto \exp\Big(\frac{\beta}{n} \sum_{x < y}
\sigma(x)\sigma(y)\Big)~.\end{equation}

The \emph{heat-bath Glauber dynamics} for the distribution $\mu_n$
is the following Markov Chain, denoted by $(X_t)$. Its state space
is $\Omega$, and at each step, a vertex $x \in V$ is chosen
uniformly at random, and its spin is updated as follows. The new
spin of $x$ is randomly chosen according to $\mu_n$ conditioned on
the spins of all the other vertices. It can  easily be shown that
$(X_t)$ is an aperiodic irreducible chain, which is reversible
with respect to the stationary distribution $\mu_n$.

We require several definitions in order to describe the
mixing-time of the chain $(X_t)$. For any two distributions
$\phi,\psi$ on $\Omega$, the \emph{total-variation distance} of
$\phi$ and $\psi$ is defined to be
$$\|\phi-\psi\|_\mathrm{TV} = \sup_{A \subset\Omega} \left|\phi(A) - \psi(A)\right| = \frac{1}{2}\sum_{\sigma\in\Omega} |\phi(\sigma)-\psi(\sigma)|~.$$
The (worst-case) total-variation distance of $(X_t)$ to
stationarity at time $t$ is
$$ d_n(t) = \max_{\sigma \in \Omega} \| \P_\sigma(X_t \in \cdot)- \mu_n\|_\mathrm{TV}~,$$
where $\P_\sigma$ denotes the probability given that $X_0=\sigma$.
 The total-variation \emph{mixing-time} of $(X_t)$, denoted by $\tmix(\epsilon)$ for $0 < \epsilon < 1$, is defined to be
$$ \tmix(\epsilon) = \min\left\{t : d_n(t) \leq \epsilon \right\}~.$$
A related notion is the spectral-gap of the chain, $\gap =
1-\lambda$, where $\lambda$ is the largest absolute-value of all
nontrivial eigenvalues of the transition kernel.

Consider an infinite family of chains $(X_t^{(n)})$, each with its
corresponding worst-distance from stationarity $d_n(t)$, its
mixing-times $\tmix^{(n)}$, etc. We say that $(X_t^{(n)})$
exhibits \emph{cutoff} iff for some sequence $w_n =
o\big(\tmix^{(n)}(\frac{1}{4})\big)$ we have the following: for
any $0 < \epsilon < 1$ there exists some $c_\epsilon > 0$, such
that
\begin{equation}\label{eq-cutoff-def}\tmix^{(n)}(\epsilon) - \tmix^{(n)}(1-\epsilon) \leq c_\epsilon w_n \quad\mbox{ for all $n$}~.\end{equation}
That is, there is a sharp transition in the convergence of the
given chains to equilibrium at time
$(1+o(1))\tmix^{(n)}(\frac{1}{4})$. In this case, the sequence
$w_n$ is called a \emph{cutoff window}, and the sequence
$\tmix^{(n)}(\frac{1}{4})$ is called a \emph{cutoff point}.

\begin{figure}[t]
\centering
\begin{tabular}{|cc|}
\hline  &\\
%\noalign\medskip
\multicolumn{2}{|c|}
{\includegraphics[width=4in]{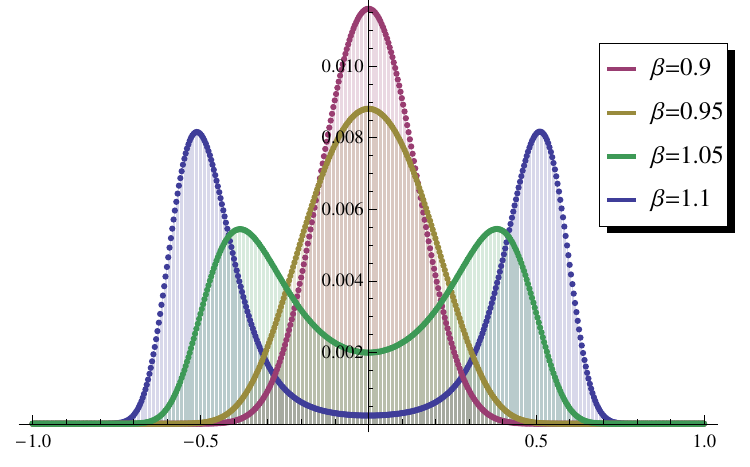}}\\
\multicolumn{2}{|c|}{(a)}\\
\hline &\\
\includegraphics[width=2in]{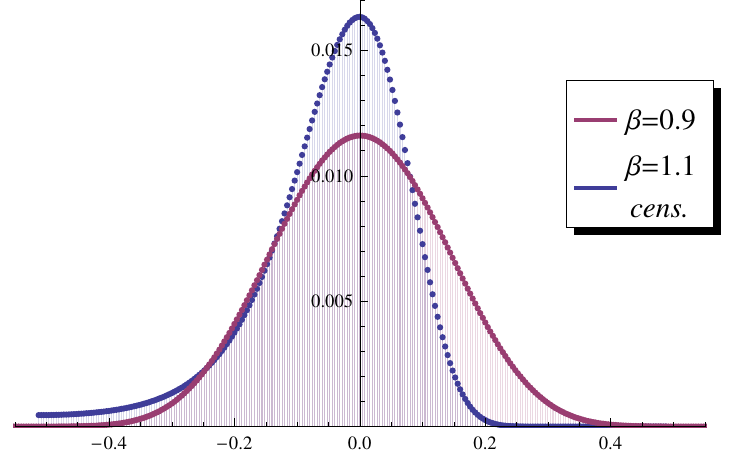}&
\includegraphics[width=2in]{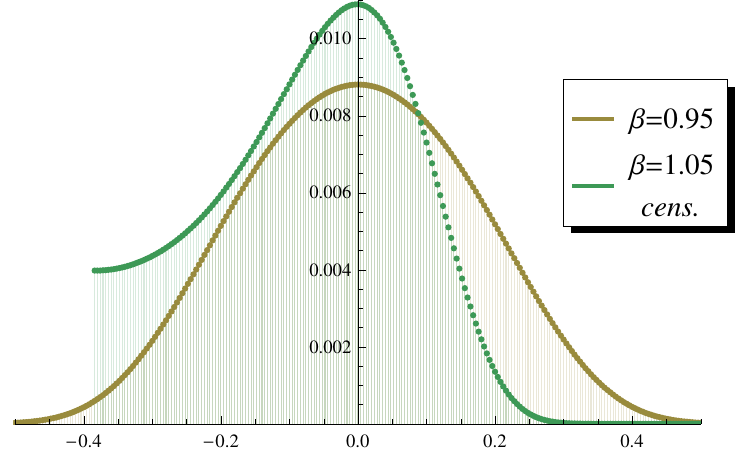}\\
\multicolumn{2}{|c|}{(b)}\\
\hline
\end{tabular}
\caption{The analogy between the original dynamics at high temperature ($\beta=1-\temp$) and the censored dynamics at low temperature ($\beta=1+\temp$).
(a) The stationary distribution of the normalized magnetization chain (average of all spins) for the original dynamics on $n=500$ vertices.
(b) The above distribution for $\beta=1-\temp$ vs. the
corresponding distribution for $\beta=1+\temp$ in the censored
dynamics, shifted by $\zeta$ (the unique
positive solution of $\tanh(\beta x) = x$).}
\label{fig:mag-stat}
\end{figure}

It is well known that for any fixed $\beta>1$, the mixing-time of the Glauber
dynamics $(X_t)$ is exponential in $n$ (cf., e.g.,
\cite{GWL}), whereas for any fixed $\beta < 1$ (high temperature)
this mixing-time has order $n\log n$ (see \cite{AH} and also
\cite{BD}). In 2007, Levin, Luczak and Peres \cite{LLP}
established that the mixing-time at the critical point $\beta_c=1$
has order $n^{3/2}$, and that for fixed $0<\beta<1$ there is
cutoff at time $\frac{1}{2(1-\beta)}n\log n$ with window $n$.
In a companion paper \cite{DLP}, we extended these results into a
complete characterization of the mixing time of the dynamics as a
function of the temperature, as it approaches its critical
point. In particular, we found a scaling window of order
$1/\sqrt{n}$ around the critical temperature. In the high
temperature regime, $\beta = 1 - \temp$ for some $0 < \temp < 1$
so that $\temp^2 n \to\infty$ with $n$, the mixing-time has order
$(n/\temp)\log(\temp^2 n)$, and exhibits cutoff with constant
$\frac{1}{2}$ and window size $n/\temp$.
 In the critical window, $\beta = 1\pm \temp$ where $\temp^2 n$ is $O(1)$, there is no cutoff, and the mixing-time has order $n^{3/2}$.
 At low temperature, there is no cutoff, and the mixing time has order
$\frac{n}{\temp}\exp\left(\frac{n}{2}\int_{0}^{\zeta} \log
\left(\frac{1+g(x)}{1-g(x)}\right)dx \right)$,
  where $g(x)=\frac{\tanh(\beta x)-x}{1-x\tanh(\beta
  x)}$ and $\zeta$ is the unique positive root of $g(x)$.

The key element in the proofs of the above results is
understanding the behavior of the sum of all spins (known as the
\emph{magnetization}) at different temperatures. This
function of the dynamics turns out to be an ergodic Markov chain
as well, namely a \emph{birth-and-death} chain (a
1-dimensional chain that only permits moves between neighboring
positions). In fact, the exponential
mixing at low-temperature is essentially due to this
chain having two centers of mass, $\pm \zeta n$, with an exponential
commute time between them.

Interestingly, this bottleneck between the two centers of mass at $\pm \zeta n$ is essentially
the \emph{only} reason for the exponential mixing-time at low
temperatures. Indeed, as shown in \cite{LLP} for the strictly
supercritical regime (the case of $\beta
>1$ fixed), if one restricts the Glauber dynamics to non-negative
magnetization (known as the \emph{censored} dynamics), the
mixing time becomes $\Theta(n\log n)$ just like in the subcritical
regime. Formally, the censored dynamics is defined as follows: at
each step, a new state $\sigma$ is generated according to the
original rule of the Glauber dynamics, and if a negative
magnetization is reached ($\sum_i \sigma(i) < 0$) then $\sigma$ is
replaced by $-\sigma$. It turns out that this simple modification
suffices to boost the mixing-time back to order $n\log n$, just as
in the high temperature case. It is thus natural to ask
whether the analogy between the original dynamics at high temperatures and the
censored one at low temperatures carries on to the existence of
cutoff.

In this work, we strengthen the above result of \cite{LLP} by showing
that the censored dynamics exhibits cutoff at low temperature beyond the
critical window, with the same order as its high temperature counterpart.

\begin{maintheorem}\label{thm-low-temp}
  Let $\temp > 0$ be such that $\temp^2 n\to\infty$ arbitrarily slowly with $n$.
  The Glauber dynamics for the mean field Ising model
   with parameter $\beta=1+\temp$,
   restricted to non-negative magnetization,
    has a cutoff at $$t_n = \left(\frac{1}{2}+\frac{1}{2(\zeta^2 \beta/\temp - 1)}\right)\frac{n}{\temp}\log(\temp^2 n)$$
  with a window of order $n/\temp$. In the special case of the dynamics started from the all-plus configuration,
  the cutoff constant is $[2(\zeta^2 \beta/\temp - 1)]^{-1}$ (the order of the cutoff point and the window size remain the same).
\end{maintheorem}

As pointed out in \cite{DLP}, the censored dynamics has a mixing-time of order $n^{3/2}$ within the critical window $\beta=1\pm\temp$ where $\temp=O(1/\sqrt{n})$. Thus, the above theorem demonstrates the smooth transition of this mixing-time from $\Theta(n^{3/2})$ to $\Theta(n\log n)$ as $\beta$ increases. Furthermore, combining this theorem with the above mentioned results of \cite{DLP} shows that the cutoff for the censored dynamics at $\beta=1+\temp$ has precisely the same order as its high temperature counterpart $1-\temp$ in the original dynamics, yet with a different constant. This analogy is illustrated in Figure \ref{fig:mag-stat}, which compares
the stationary distribution of the two corresponding magnetization chains.

In addition, we determine the spectral gap for the censored dynamics at low temperatures, which again
proves to have the same order as in the high temperature regime of the original dynamics.

\begin{maintheorem}\label{thm-low-temp-spectral}
  Let $\temp>0$ be such that $\temp^2 n\to\infty$ arbitrarily slowly with $n$.
  Then the censored Glauber dynamics for the mean field Ising model
   with parameter $\beta=1+\temp$
    has a spectral gap of order $\temp /n$.
\end{maintheorem}

%It is worth noting that, given its fast mixing with cutoff,
% the censored dynamics comprises an efficient method of sampling
%the Gibbs distribution at low temperatures: one can simply run
%this dynamics until exceeding the cutoff point, then toss a fair coin
%to possibly flip all spins.

The rest of the paper is organized as follows. Section
\ref{sec:outline} outlines the main ideas of the proofs for the
main theorems. Several preliminary facts on the Curie-Weiss model
are introduced in Section \ref{sec:prelim}. Section \ref{sec:mag}
contains a delicate analysis of the behavior of the censored magnetization chain
for the case  $\temp=o(1)$. Based on the results of this section, we establish the
cutoff of the dynamics (Theorem \ref{thm-low-temp}) in Section
\ref{sec:fullmixing}, and determine the spectral gap (Theorem
\ref{thm-low-temp-spectral}) in Section \ref{sec:spectral}.
Section \ref{sec:temp-fixed} contains the
modifications required to prove the (simpler) case where
 $\temp$ is fixed. The final section, Section \ref{sec:conclusion},
 is devoted to concluding remarks and some open problems.

%\begin{maintheorem}\label{thm-critical-temp}
%  Let $\temp =\temp(n)$ satisfy $\temp = O(n^{-1/2})$. The mixing time of the restricted Glauber dynamics for the Ising model on $K_n$ with parameter $\beta=1+\temp$ satisfies $\tmix = \Theta(n^{3/2})$.
%\end{maintheorem}

\section{Outline of proofs and mains ideas}\label{sec:outline}
In this section, we outline the proofs of the main theorems and
highlight the main ideas and techniques required to prove the case
where $\temp =o(1)$ (the proofs for the $\temp$ fixed case follow
the same line of arguments).

\begin{figure}
\centering \fbox{\includegraphics[width=4.5in]{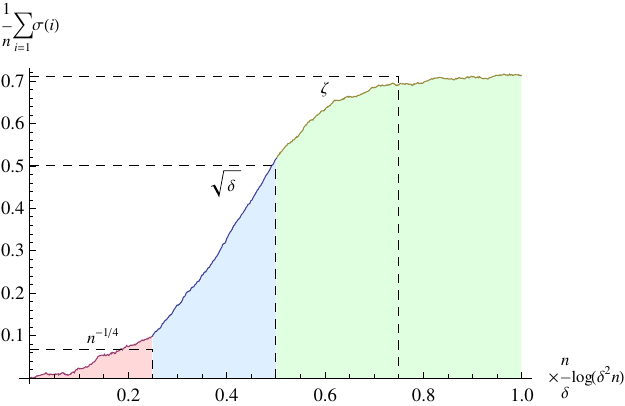}}
\caption{The magnetization chain in a single simulation of the censored dynamics on $n=50,000$ vertices
at temperature $\beta=1.25$. Results show the three segments of the magnetization started at $s_0=0$, until
it reaches equilibrium near $\zeta$ (the unique positive root of $g(x)=\tanh(\beta x)-x$).}
\label{fig:mag-sim}
\end{figure}

\subsection{Cutoff of the magnetization chain}
Clearly, in order to obtain the mixing of the entire Glauber dynamics,
it is necessary to achieve the mixing of its magnetization. Hence,
 we first study the normalized censored magnetization
chain, $\cS_t \deq \frac{1}{n}\sum_i \cX_t(i)$, where $\cX_t$ denotes
the configuration of the censored dynamics at time $t$. It turns out that the stationary distribution of $\cS_t$
 concentrates around $\zeta$ at low temperatures. Therefore, we need to show that, for any starting
position, the magnetization will hit near $\zeta$ around the
cutoff point. To show this, we consider the two extreme cases: starting
from $\cS_0=0$ and starting from $\cS_0=1$.

The case $\cS_0=1$ is significantly simpler,
and follows basically from the contraction properties of the magnetization chain. However,
the case $\cS_0=0$ requires a delicate analysis. As mentioned in the introduction, in order to obtain the concentration of the hitting time from $0$ to $\zeta$, we partition the region $[0,\zeta]$ into three segments: $[0,n^{-1/4}]$,
 $[n^{-1/4},\sqrt{\temp}]$ and $[\sqrt{\temp},\zeta]$ (up to constants). Figure \ref{fig:mag-sim} shows the
transition of the magnetization chain between these three segments, as it occurred in a sample run of the censored dynamics.

 In each of the three segments, we exploit different properties of the magnetization chain to track its position along time. As we later show, the properties of the Hyperbolic tangent function dominate the behavior
of the magnetization chain. Around $0$, the function $\tanh (\beta x)$ is well approximated
by a linear function, which in turn leads to an exponential growth in the expected value of the
magnetization near
$0$. Around $\zeta$, the Taylor expansion of $\tanh(\beta x)$
implies that the magnetization is contracting towards $\zeta$.

In order to achieve the concentration of $\cS_t$, we introduce the times $T_i^+$ and $T_i^-$ for $i=1,2,3,4$,
where the difference between $T_i^+$ and $T_i^-$ is $O(n/\temp)$, hence can be absorbed in the cutoff window.
These times correspond to the above three segments together with the segment $[\zeta,1]$ (which treats the case
$\cS_0=1$), and we study the position of $\cS_t$ in each of them.

\subsubsection*{Hitting $n^{-1/4}$ from $0$}

% While we believe these two phases are
%substantial for travelling from $0$ to $\zeta$,  we divide the
%first phase into two such that we can have a better control of the
%variability of the magnetization by resetting our starting
%position.

This segment begins with a ``burn-in'' period of $O(n/\temp)$ steps,
which is in fact the only regime where we benefit from the censoring of the
dynamics. By the end of this burn-in period, $\cS_t$ will have hit position $1/\sqrt{\temp n}$
with probability arbitrarily close to $1$.
%by noting that $A / \sqrt{\temp n} -
%\cS_t$ is a supermartingale with substantial variance at each
%step, and applying a suitable supermartingale lemma.
Once the magnetization reaches order $1/\sqrt{\temp n}$, we may
analyze the effect of the exponential growth of its expected value
(dictated by the above mentioned properties of the Hyperbolic tangent function).

Two elements are needed in order to complete the analysis of this segment.
First, we we establish an upper bound on $\var \cS_t$.
Second, we carefully bound the difference
between the $\E \cS_t^3$ and $(\E \cS_t)^3$, which allows us to switch these two when
tracking down the slight changes in $\E \cS_t$ along time (via the Taylor expansion of $\tanh(\beta x)$ in this regime).

Altogether, we show that with probability arbitrarily close to $1$, we have $\cS_{T_1^-} \leq n^{-1/4}$ and yet $\cS_{T_1^+} \geq n^{-1/4}$.

\subsubsection*{Hitting $\sqrt{\temp}$ from $n^{-1/4}$}
Given a starting position of $n^{-1/4}$, an analogous argument which
tracks $\E\cS_t$ (using the exponential growth given by the Taylor expansion of the Hyperbolic tangent around $0$)
implies that with high probability, $\cS_{T_2^+} \geq  \frac43\sqrt{\temp} $
and yet $\cS_{T_2^-} \leq \sqrt{\temp}$. Crucially, though the above argument is similar to the one used for the previous segment, resetting the starting position to $n^{-1/4}$ (by separating the treatment of the first two segments)
provides the required control over the variability of $\cS_t$.

\subsubsection*{Hitting $\zeta$ from $\sqrt{\temp}$}
Given a starting position of, say, $\frac{4}{3}\sqrt{\temp}$, with high probability $\cS_t$ will remain
above, say, $\frac{7}{6}\sqrt{\temp}$ for at least $T_3^+$ steps.
In this region, the magnetization is attracted towards $\zeta$,
and combing this with correlation inequalities (e.g., the FKG inequality) one can
obtain the bound $\var \cS_t = O(n/\temp)$. Altogether, we show that with high probability
$|\cS_{T_3^+} - \zeta |$ is at most $O(1/\sqrt{\temp n})$ whereas
$\cS_{T_3^-}$ is further below $\zeta$.

\medskip
The results for the above three segments establish cutoff of the magnetization chain started at $\cS_0=0$. To complete the analysis, we treat the case $\cS_0=1$ in the fourth segment described next.

\subsubsection*{Hitting $\zeta$ from $1$}
Starting from $\cS_0=1$, the magnetization is strongly attracted
towards $\zeta$. In fact, its behavior throughout this segment is roughly
equivalent to that in the segment $[\sqrt{\temp},\zeta]$, and as a result, the expected hitting
time from $1$ to $\zeta$ is asymptotically the same as that from $\sqrt{\temp}$ to $\zeta$
(explaining the relation between the two cutoff constants in Theorem \ref{thm-low-temp}).
To show this, we obtain a variance bound, analogous to the one derived in
the segment $[\sqrt{\temp},\zeta]$, and deduce that with high probability
$|\cS_{T_4^+} - \zeta |$ is at most $O(1/\sqrt{\temp n})$ whereas
$\cS_{T_4^-}$ is further above $\zeta$.

\subsubsection*{Coalescence of the censored magnetization chains} To establish an upper bound on the mixing-time of the censored magnetization, we construct a coupling of two instances of the censored
Glauber dynamics, which ensures a fast collision. There are
three key points in accomplishing this coupling.
\begin{enumerate}[(i)]
\item Around the cutoff point, with probability
arbitrarily close to 1, the magnetization is concentrated around $\zeta$
within distance $O(1/\sqrt{\temp n})$.
\item Starting from
somewhere near $\zeta$, with high probability the magnetization chain will stay ``sufficiently close'' to
$\zeta$ for a reasonably long period of time: Within this distance from $\zeta$, the magnetization demonstrates
certain contraction properties, and we can use correlation inequalities (such as the FKG inequality) to control its higher moments.
\item In the above mentioned contracting region, $|\cS_t - \zeta|$
behaves as a supermartingale with a non-negligible variance at each
step. Altogether, we can deduce that within $O (n/\temp)$ steps
beyond the cutoff point, the two censored magnetization chains will collide
with probability close to 1.
\end{enumerate}
Combining the above coupling argument with the behavior of the stationary distribution of the
 censored magnetization (which concentrates around $\zeta$), as well as the lower bounds we obtained for hitting $\zeta$ from $\cS_0=0$ or $\cS_0=1$, completes the proof of the magnetization cutoff.

\subsection{Full mixing of the dynamics}
The cutoff point of the censored magnetization chain
clearly gives a lower bound on the mixing-time of the entire dynamics.
Furthermore, note that in the special case where the dynamics starts from the
all-plus configuration, by symmetry it has a cutoff precisely whenever
the magnetization chain exhibits cutoff. It remains to generalize
this result to an arbitrary starting configuration.
To boost the mixing of the censored magnetization (from an arbitrary starting
position) to the mixing
of the full dynamics, we use a Two Coordinate Chain
analysis, following the approach of \cite{LLP}. In order to apply this
method, one needs to establish a series of
delicate conditions on the censored magnetization chain.

First, we combine an expectation analysis with
concentration arguments to show that after $n/\temp$ steps,  with high probability
the censored magnetization starting from $\cS_0=1$ will stay at some ``good state" -- roughly, not too biased towards
plus or minus. As a corollary
(since the all-plus initial position can be used to sandwich the remaining initial positions),
this holds for any starting position $\cS_0$.

Two additional conditions are required to complete the Two Coordinate Chain analysis.
First, we show that the censored magnetization almost surely stays around $\zeta$ for a sufficiently
long period beyond its cutoff point. Second, we show that with high probability,
the average value of a spin over the set of initially positive spins (i.e., $\{i: \sigma_0(i) =1\}$)
also concentrates around $\zeta$ for a reasonably long period.

These properties imply that the magnetization restricted to the set of initially positive spins
mixes at the cutoff point, and the same holds for the magnetization over the set of initially negative
spins. By symmetry, these two statements imply the entire mixing of the dynamics.

\subsection{Spectral gap analysis}
We first study the spectral gap of the censored magnetization chain,
which provides an immediate upper bound on the spectral gap of the
entire dynamics.
To determine this gap, we analyze the conductance of the
chain (a birth-and-death chain) following the approach of \cite{DLP}*{Section 6},
and establish the order of the bottleneck ratio, yielding an
effective lower bound. To obtain a matching upper bound, we use
the Dirichlet representation for the spectral gap, combined with
an appropriate bound on the \emph{fourth} central moment of the
censored magnetization in stationarity.

To infer the spectral gap of the full dynamics from that of the
censored magnetization, additional arguments are needed to obtain a
lower bound on the gap. We separate the eigenfunctions into two
orthogonal spaces, one of which exactly corresponds to the censored
magnetization chain. We then use the contraction properties of the dynamics
to prove that on the other space, the corresponding eigenvalues are
uniformly bounded from above. This implies the desired lower bound
for the spectral gap.

\section{Preliminaries} \label{sec:prelim}

\subsection{Magnetization chain and censored magnetization chain}\label{subsec:magnetization}
In our efforts to analyze the censored Glauber
dynamics, in many cases it is useful to study the original dynamics
and relate it to the censored one. Throughout the paper,
we let $X_t$, $S_t$ denote the original Glauber dynamics and
its corresponding magnetization chain, and let $\cX_t$ and
$\cS_t$ denote the censored dynamics and its magnetization
chain.

Recall that the normalized magnetization of a configuration $\sigma$
is defined as $S(\sigma) \deq \frac{1}{n} \sum_{j=1}^{n} \sigma(j)$
(we define $\cS(\sigma)$ analogously for the censored dynamics).
  In the \emph{original} Glauber dynamics, given that the current state
of the dynamics is $\sigma$ and a site $i$ has been selected for
updating, the probability of updating $i$ to a positive spin is
$p^{+}(S (\sigma)- n^{-1}\sigma(i))$, where $p^{+}$ is the
function given by
\begin{equation*}%\label{eq-def-p-+}
p^{+}(s)\deq\frac{e^{\beta s}} {e^{\beta s} + e^{-\beta s}}=\frac{1 + \tanh(\beta s)}{2}~.
\end{equation*}
 Similarly, with probability $p^{-}(S (\sigma)- n^{-1}\sigma(i))$ site $i$ is updated to a negative spin,
where $p^{-}$ is the function given by
\begin{equation*}%\label{eq-def-p--}
 p^{-}(s)\deq \frac{e^{-\beta s}}{e^{\beta s} + e^{-\beta s}}= \frac{1-\tanh(\beta s)}{2}~.
\end{equation*}
We can then obtain the transition probabilities of the original
magnetization chain:
\begin{align}\label{eq-magnet-transit}
P_M(s, s')=
\begin{cases}
\frac{1-s}{2}p^{-}(s-n^{-1}) & \hbox{ if } s' = s+\frac{2}{n}, \\
\frac{1+s}{2}p^{+}(s+ n^{-1}) & \hbox{ if } s'= s- \frac{2}{n},\\
1-\frac{1-s}{2}p^{-}(s-n^{-1})- \frac{1+s}{2}p^{+}(s+ n^{-1}) & \hbox{ if } s'=s~.
\end{cases}
\end{align}
It is easy to verify that, by definition, the \emph{censored} magnetization chain $\cS_t$ has the same distribution
law as $|S_t|$, and hence has the following transition matrix
$\cP_\textsf{M}$:
\begin{equation}\label{eq-magnet-cens-transit}
\cP_\textsf{M}(s, s')= P_M(s, s') + P_M (s,-s')~.
\end{equation}
The next lemma will prove to be useful in the analysis of the censored
magnetization chain.
\begin{lemma}[\cite{LPW}*{Chapter 17}]\label{lem-supermatingale-positive}
  Let $(W_t)_{t \geq 0}$ denote a non-negative
  supermartingale and $\tau$ be a stopping time such
  \begin{enumeratei}
    \item $W_0 = k$,
    \item $W_{t+1} - W_t \leq B$,
    \item $\var(W_{t+1} \mid \F_t) > \sigma^2 > 0$
      on the event $\tau > t$ .
  \end{enumeratei}
  If $u > 4 B^2/(3 \sigma^2)$, then $\P_k( \tau > u) \leq \frac{ 4 k }{\sigma \sqrt{u}}$.
\end{lemma}

\subsection{Monotone coupling}\label{sec:monotone coupling} A useful tool throughout our arguments is the \emph{monotone coupling} of two instances of the Glauber dynamics $(X_t)$ and $(\tX_t)$, which maintains a coordinate-wise inequality between the corresponding configurations. That is, given two configurations $\sigma \geq \tilde{\sigma}$ (i.e., $\sigma(i)\geq \tilde{\sigma}(i)$ for all $i$), it is possible to generate the next two states $\sigma'$ and $\tilde{\sigma}'$ by updating the same site in both, in a manner that ensures that $\sigma' \geq \tilde{\sigma}'$. More precisely, we draw a random variable $I$ uniformly over
$\{1, 2, \ldots, n\}$ and independently draw another random variable $U$ uniformly over $[0,1]$. To generate
$\sigma'$ from $\sigma$, we update site $I$ to $+1$ if $U \leq p^{+}\left(S(\sigma)-\frac{\sigma(I)}{n}\right)$,
 otherwise $\sigma'(I)=-1$. We perform an analogous process in order to generate $\tilde{\sigma}'$ from
$\tilde{\sigma}$, using the same $I$ and $U$ as before. The monotonicity of the function $p^{+}$ guarantees that
$\sigma' \geq \tilde{\sigma}'$, and by repeating this process, we obtain a coupling of the two instances of the
 Glauber dynamics that always maintains monotonicity.
Clearly, this coupling induces a monotone coupling for the two corresponding magnetization chains.

We say that a birth-and-death chain with a transition kernel $P$ and a state-space $\magspace=\{0,1,\ldots,n\}$ is
\emph{monotone} if $P(i,i+1) + P(i+1,i) \leq 1$ for every $i < n$. It is easy to verify that this condition is
equivalent to the existence of a monotone coupling between two instances of the chain. Hence, by the above discussion, the magnetization chain $S_t$ is indeed a monotone birth-and-death chain.

%Projecting the configurations into magnetization in the above
%coupling, we immediately obtain a monotone coupling for the
%magnetization chain.
In addition, we will also need a monotone
coupling for the censored magnetization chain $\cS_t$. The only
questionable point is the state nearest to 0. Assuming that $n$ is even
 (the case where $n$ is odd follows from the same argument), this
question is reduced to the following: taking $\cS_0 = \frac{2}{n}$
and $\tilde{\cS}_0= 0$, can we construct a coupling such that
$\cS_1 \geq \tilde{\cS}_1$. This is indeed guaranteed by the fact that $
\cP_\textsf{M}(0, \frac{2}{n}) + \cP_\textsf{M} (\frac{2}{n}, 0) \leq
1$, hence the censored magnetization chain $\cS_t$ is monotone as well.

Note that there does \emph{not exist} a monotone coupling for the censored
Glauber dynamics. To see this, consider the case of $n$ even. Let $\sigma$ be
a configuration with $\cS(\sigma)=0$, and let $\tilde{\sigma}$ be a configuration
which differs from $\sigma$ in precisely one coordinate $i$ where $\sigma(i)=-1$.
Next, consider two instances of the censored Glauber
dynamics $\cX_t$ and $\tilde{\cX}_t$ starting from $\sigma$ and $\tilde{\sigma}$ resp.
By definition of the censored dynamics, with positive probability $\cX_1$
will flip $n-1$ spins, including all $n/2$ spins that were negative in $\sigma$.
Thus, in order to maintain monotonicity, $\tilde{\cX}_1$ must in this case update $\frac{n}{2}-1$
sites from minus to plus. However, the 1-step censored Glauber dynamics started from
$\tilde{\sigma}$ is exactly the same as the original Glauber
dynamics, where only one spin can be updated. We conclude that no monotone
coupling exists.

\section{Cutoff for the magnetization chain}\label{sec:mag}
The goal of this section is to establish cutoff for the censored magnetization
chain $\cS_t$, as stated in the following theorem:
\begin{theorem}\label{thm-cutoff-mag}
Let $\beta = 1+ \temp$, where $\temp > 0 $ satisfies $\temp^2 n
\to \infty$. Then the corresponding censored magnetization chain
$(\cS_t)$ exhibits cutoff at time $$t_n =
\left(\frac{1}{2}+\frac{1}{2(\zeta^2 \beta/\temp -
1)}\right)\frac{n}{\temp}\log(\temp^2 n)$$
 with a window of order $n/\temp$. In the special case $\cS_0=1$ (starting from the all-plus configuration), the cutoff has the same order of mixing-time and window, yet its constant is $[2(\zeta^2 \beta/\temp - 1)]^{-1}$.
\end{theorem}

The next simple lemma, which appeared in \cite{DLP} and is a special case of a lemma of \cite{LLP}, illustrates the importance of the magnetization chain. We include its proof for completeness.
\begin{lemma}[\cite{DLP}*{Lemma 3.2}]\label{lem-all-plus-mag-full}
Let $(\cX_t)$ be an instance of the censored Glauber dynamics for
the mean field Ising model starting from the all-plus configuration,
namely, $\sigma_0 = \mathbf{1}$, and let $\cS_t = S(\cX_t)$ be its
magnetization chain. Then
\begin{equation}
  \label{eq-allplus-xt-st-equiv}
  \| \P_{\mathbf{1}}(\cX_t \in \cdot)- \mu_n\|_\mathrm{TV} = \| \P_{\mathbf{1}}(\cS_t \in \cdot)- \pi_n\|_\mathrm{TV}~,
\end{equation}
where $\pi_n$ is the stationary distribution of the censored
magnetization chain.
\end{lemma}
\begin{proof}
For any $s \in \{0, \frac{2}{n},\ldots,1-\frac{2}{n},1\}$, let
$\Omega_s \deq \{\sigma \in \Omega: S(\sigma) = s\}$. Since by
symmetry, both $\mu_n(\cdot \mid \Omega_s)$ and
$\P_{\mathbf{1}}(\cX_t \in \cdot \mid \cS_t = s)$ are uniformly
distributed over $\Omega_s$, the following holds:
\begin{align*}\|\P_{\mathbf{1}}(\cX_t\in\cdot)-\mu_n\|_\mathrm{TV} &= \frac{1}{2}\sum_{s} \sum_{\sigma\in\Omega_s} \left|\P_{\mathbf{1}} (\cX_t = \sigma) - \mu_n(\sigma)\right| \\
&= \frac{1}{2}\sum_s \sum_{\sigma\in\Omega_s}
\Big|\frac{\P_\mathbf{1} (\cS_t = s)}{|\Omega_s|}
 - \frac{\mu_n(\Omega_s)}{|\Omega_s|}\Big| \\
&= \|\P_\mathbf{1}(\cS_t\in\cdot)-\pi_n\|_\mathrm{TV}~.\qedhere
\end{align*}
\end{proof}

Combining Theorem \ref{thm-cutoff-mag} with the above lemma immediately establishes cutoff for the censored dynamics starting from the all-plus configuration.
\begin{corollary}
Let $\temp > 0$ be such that $\temp^2 n \to \infty$, and let
$(\cX_t)$ denote the censored Glauber dynamics for the mean-field
Ising model with parameter $\beta = 1 + \temp$, started from
all-plus configuration. Then $(\cX_t)$ exhibits cutoff at time
$[2(\zeta^2 \beta/\temp - 1)]^{-1}\frac{n}{\temp}\log(\temp^2 n)$
  with a window of order $n/\temp$.
\end{corollary}

In order to prove Theorem \ref{thm-cutoff-mag}, we consider 4 phases for the censored magnetization chain. For each phase, we select a pair of
times, $T_i^+$ and $T_i^-$, whose difference can be absorbed into the cutoff window; we then establish
that with probability arbitrarily close to $1$, the magnetization at $T_i^-$ is smaller than some
given target value, whereas at $T_i^+$ it is larger than this value. That is, a given value is typically being
sandwiched by the magnetization at the two time-points $T_i^-$ and $T_i^+$, and this allows us to continue the analysis
with this value serving as the new starting point of the magnetization chain.
For instance, in the first phase, we start from $\cS_0=0$ and the above mentioned target value for the magnetization is $n^{-1/4}$, hence this phase is referred to as ``Getting from $0$ to $n^{-1/4}$'', and studied in Subsection \ref{subsec-0-to-n-1/4}. The remaining 3 phases appear in Subsections \ref{subsec-n-1/4-to-sqrt-temp}, \ref{subsec-sqrt-temp-to-zeta} and \ref{subsec-1-to-zeta}.

For the sake of simplicity, we assume throughout the section that
$\temp=o(1)$, as this case captures most of the difficulties in
establishing the cutoff points. Section \ref{sec:temp-fixed}
contains the changes one needs to make in order for the
proof to hold in the (simpler) case of $\temp$ fixed.

Set $\beta =
1 + \temp$ and $\temp^2 n \to \infty$. Let $\zeta$ denote the
unique positive solution to $\tanh(\beta x)=x$, and notice that
the Taylor expansion
\begin{equation}
  \label{eq-taylor-tanh}
  \tanh(\beta x) = \beta x - \frac{1}{3}(\beta
x)^3 + O((\beta x)^5)
\end{equation} implies that whenever $\temp = o(1)$ we get
$$\zeta = \sqrt{3\temp/\beta^3 - O((\beta \zeta))^5} = \sqrt{3\temp} + O(\temp^{3/2})~.$$
%Similarly, define $\gamma^\star = \beta / \cosh^2(\beta \zeta)$, and since $\cosh^{-2}(x)=1-x^2+O(x^4)$ we have
%$$\gamma^\star = \frac{\beta}{\cosh^2(\beta \zeta)} \approx \beta(1-3\temp) \approx 1-2\temp~.$$

%We would like to point out a notation issue before we jumping into the proofs. By state $s$, we mean the nearest
%state to $s$ in the magnetization space.

\subsection{Getting from $0$ to $n^{-1/4}$}\label{subsec-0-to-n-1/4}
In this subsection, we address the issue of reaching a magnetization of $n^{-1/4}$ from $\cS_0=0$.
\begin{theorem}\label{thm-reach-0-n^-1/4}
Define
\begin{eqnarray*}
  &T_1 \deq \frac14 (n/\temp)\log (\temp^2 n)~,\\
  &T_1^+ (\gamma) \deq T_1 + \gamma n/\temp\quad,\quad T_1^-(\gamma) \deq T_1 -\gamma n/\temp~.
\end{eqnarray*} The following holds for the censored magnetization chain
$\cS_t$:
\begin{align}
&\lim_{\gamma \to\infty} \liminf_{n\to\infty} \P_0(\cS_{T_1^+ (\gamma)} \geq n^{-1/4}) =1~,\label{eq-upper-bound-lemma-reach-1}\\
&\lim_{\gamma \to\infty} \limsup_{n\to\infty} \P_0(\cS_{T_1^-
(\gamma)} \geq n^{-1/4})=0~. \label{eq-lower-bound-lemma-reach-1}
\end{align}
\end{theorem}

\subsubsection{Proof of \eqref{eq-upper-bound-lemma-reach-1}: Lower bound of $n^{-1/4}$ for $\cS_{T_1^+}$}
To establish the mentioned lower bound on $\cS_{T_1^+}$, we first show that within some negligible burn-in period,
the censored magnetization chain $\cS_t$ will hit near $A/\sqrt{\temp n}$.
\begin{lemma}\label{lem-hit-A-good-state}
There exists some constant $c > 0$ such that the following holds: For any $A,\gamma > 0$, the censored magnetization chain $\cS_t$
started at $\cS_0\in\{0,\frac1n\}$ will hit $A/\sqrt{\temp n}$ within
$\gamma n /\temp$ steps with probability at least $1- c A /\sqrt{\gamma}$.
\end{lemma}
\begin{proof}
The transition probabilities of the censored
magnetization chain, as given in  \eqref{eq-magnet-transit} and \eqref{eq-magnet-cens-transit},
together with the fact that $\tanh (\beta s) \geq s$ for $0\leq s\leq
\zeta$, imply that $\cS_t$ is a
non-negative submartingale. Thus, $A/\sqrt{\temp n} -
\cS_t$ is immediately a supermartingale. Recalling that the holding
probability for the magnetization chain is bounded uniformly from
below and above, we infer that the conditional variance at each
step is bounded uniformly from below. Therefore, upon defining $$\tau_{A /\sqrt{\temp
n}} \deq \min \{t: \cS_t \geq A /\sqrt{\temp n}\}$$
we may apply Lemma \ref{lem-supermatingale-positive} and obtain that for some
absolute constant $c> 0$,
\begin{equation*}
\P_{0}\left(\tau_{A /\sqrt{\temp n}} \geq \gamma
\frac{n}{\temp}\right) \leq \frac{c A}{\sqrt{\gamma}}~. \qedhere
\end{equation*}
\end{proof}

Next, we can assume that the chain is started from $A /\sqrt{\temp
n}$. With this assumption, we can simply approach the censored
magnetization chain $\cS_t$ via the original magnetization chain
$S_t$, as shown in the following.

We first establish an upper bound for the variance of the
magnetization.
\begin{lemma}\label{lem-var-bound-0-n-1/4}
Let $(S_t)$ be a magnetization chain with some arbitrary starting position $s_0$. Then for some absolute constant $c > 0$, the following holds:
\begin{equation}
  \label{eq-var-st-bound}
  \var_{s_0} S_t \leq \frac{5}{n^2} \sum_{j=0}^{t-1} \left(1+\frac{2\temp}{n}\right)^j \leq
  \frac{c}{\temp n}\left(1+\frac{2\temp}{n}\right)^t~.
\end{equation}
\end{lemma}
\begin{remark*} Unlike the high temperature regime, where (using the contraction property of the magnetization chain) the variance can be uniformly bounded from above for all
$t$, the above bound on the variance grows with $t$. Although this bound is not sharp, it will suffice for our purposes.
\end{remark*}
\begin{proof}
The censored magnetization chain does not exhibit contraction properties in the low temperature regime, and so our argument will follow from tracking the change in the variance after each additional step. To this end, we
first establish two recursion relations, for $(\E S_t)^2$ and $\E S_t^2$ respectively. By \eqref{eq-magnet-transit}, we get
that
\begin{align}
&\E \left [ S_{t+1} \mid S_t = s \right ] \nonumber\\
&= \left(s+\frac{2}{n}\right)P_M\left(s, s+\frac{2}{n}\right) + sP_M(s,s) + \left(s-\frac{2}{n}\right)P_M\left(s, s-\frac{2}{n}\right) \nonumber\\
&= \left(1+\frac{\temp}{n}\right)s + \frac{1}{n}\left(\tanh(\beta
s) - \beta
s\right)-\left|O\left(\frac{s}{n^2}\right)\right|~.\label{eq-St-1st-moment-tanh}
\end{align}
Taking expectation and squaring, we obtain that
\begin{align}\label{eq-exp-St-square}
(\E S_{t+1})^2 &\geq \left(1 + \frac{2\temp}{n}\right)(\E S_t)^2 +
\frac{2}{n}\E \left(\tanh(\beta S_t) - \beta S_t\right) \E S_t +
\frac{c}{n^2} ~.
\end{align}
Applying an analogous analysis onto the second moment yields
\begin{align}
  \E&\left[S_{t+1}^2 \mid S_t = s\right] = s^2 + \frac{2s}{n}\left(p^+(s-n^{-1}) - p^-(s+n^{-1})\right)\nonumber\\
  &-\frac{2s^2}{n}\left(p^-(s-n^{-1})+p^+(s+n^{-1})\right) \nonumber\\
  &+ \frac{4}{n^2}\left(\frac{1+s}{2}p^-(s-n^{-1})
  + \frac{1-s}{2}p^+(s+n^{-1})\right)~.\nonumber
\end{align}
Since
\begin{align*} p^-(s-\mbox{$\frac1n$})+p^+(s+\mbox{$\frac1n$}) &= \frac{1}{2}\left(2+\tanh\left(\beta(s+\mbox{$\frac1n$})\right)
-\tanh\left(\beta(s-\mbox{$\frac1n$})\right)\right) \\
&= 1 + \frac{1}{n \cosh^2(\xi)}\end{align*} for some $\beta(s-n^{-1}) < \xi <\beta(s+n^{-1})$, and
$$
p^+(s-n^{-1})-p^-(s+n^{-1}) =
\frac{1}{2}\left(\tanh\left(\beta(s+n^{-1})\right)
+\tanh\left(\beta(s-n^{-1})\right)\right)~,$$ the
concavity of the Hyperbolic function gives that
\begin{align}
 \E\left[S_{t+1}^2 \mid S_t = s\right]
 %\leq s^2 + \frac{s}{n}\left(\tanh\left(\beta(s+n^{-1})\right)
%+\tanh\left(\beta(s-n^{-1})\right)-2 \beta s\right) \nonumber\\
%&+ 2\beta \frac{s^2}{n} - \frac{2}{n}s^2  + \frac{4}{n^2}\nonumber\\
&\leq s^2\left(1+\frac{2\temp}{n}\right) +
\frac{2s}{n}\left(\tanh(\beta s)-\beta s\right) +
\frac{4}{n^2}~.\label{eq-St-2nd-moment}
\end{align}
Taking expectation,
\begin{align}
 \E S_{t+1}^2 &\leq \left(1+\frac{2\temp}{n}\right)\E S_t^2 + \frac{2}{n}\E\left[S_t \left(\tanh(\beta
 S_t)-\beta S_t\right)\right] + \frac{4}{n^2}~.\label{eq-St-2nd-moment-1}
\end{align}
Crucially, we claim that the next quantity is
non-positive:
\begin{align*}
 D_t \deq &\E\left[ S_t\left(\tanh(\beta S_t)-\beta S_t \right)\right]- \E\left[\tanh(\beta S_t)-\beta S_t\right](\E S_t)~.
\end{align*}
To see this, once can verify that the function $f(s)=\tanh(\beta
s)-\beta s$ is monotone decreasing in $s$. Thus, the fact that
$D_t \leq 0$ follows from the FKG inequality, and together with \eqref{eq-exp-St-square} and \eqref{eq-St-2nd-moment}, it implies that for large $n$,
$$\var S_{t+1} \leq \left(1+\frac{2\temp}{n}\right)\var S_t + \frac{c}{n^2} ~.$$
Iterating, we obtain that
\begin{equation*}
  \var_{s_0} S_t \leq \frac{c}{n^2} \sum_{j=0}^{t-1} \left(1+\frac{2\temp}{n}\right)^j \leq
  \frac{c}{\temp n}\left(1+\frac{2\temp}{n}\right)^t~.\qedhere
\end{equation*}
\end{proof}

Another ingredient required for tracking the magnetization along time appears in the following lemma,
in the form of a bound on the difference between $(\E S_t)^3$ and $\E S_t^3$.
\begin{lemma}\label{lem-St-3rd-moment-exp-cube}
Let $W_{s_0}(t) \deq \E_{s_0}
S_t^3 - (\E_{s_0} S_t)^3$, where $(S_t)$ is the magnetization chain started
from $s_0 \geq 0$. Then for some absolute constant $c > 0$,
\begin{align*}
  W_{s_0}(t) \leq \frac{c}{\temp n} \left(s_0 + \frac1n\right) \mathrm{e}^{3t \temp / n}~.
\end{align*}
\end{lemma}
\begin{proof}
Recalling \eqref{eq-St-1st-moment-tanh}, taking expectation and
rearranging both sides, we obtain the following:
\begin{align}
\left(\E S_{t+1} \right)^3 &\geq \left(1+\frac{3\temp}{n}\right)(\E
S_t)^3 +  \frac{c'}{n^2}(\E
S_t)^2\nonumber\\ &+\frac{3}{n}\E\left[\tanh\left(\beta S_t)-\beta S_t
\right)\right] (\E S_t)^2 ~.\label{eq-1st-moment-cube-lower-bound}
\end{align}
We next establish a recursion relation for $\E S_t^3$. Recalling the transition matrix $P_M$ as given in
\eqref{eq-magnet-transit}, we have
\begin{align*}
\E \left [ S_{t+1}^3 \mid S_t = s \right ]&= \frac{1+s}{2} p^-(s -
\mbox{$\frac1n$}) \left(s
-\mbox{$\frac2n$}\right)^3+ \frac{1-s}{2} p^+(s
+ \mbox{$\frac1n$}) \left(s+\mbox{$\frac2n$}\right)^3\nonumber\\
&\hspace{.3cm}+ \left(1 - \frac{1+s}{2} p^-(s - \mbox{$\frac1n$}) - \frac{1-s}{2} p^+(s + \mbox{$\frac1n$})\right) s^3 \nonumber\\
&= s^3  +c_1\frac{s}{n^2} + \frac{c_2}{n^3}+ \frac{6s^2}{n}\cdot \frac{1}{4}\Big(-2s \\
&\hspace{.3cm}+
\tanh\left(\beta(s-n^{-1})\right)+ \tanh\left(\beta(s+n^{-1})\right)
 + 2\beta s - 2\beta s \\
 &\hspace{.3cm}+ s\left(\tanh\left(\beta(s-n^{-1})\right) - \tanh\left(\beta(s+n^{-1}\right)\right)
\Big)~,\end{align*}
Combined with the concavity of the Hyperbolic tangent, this gives
\begin{align*}\E \left [ S_{t+1}^3 \mid S_t = s \right ]
& \leq \left(1+\frac{3\temp}{n}\right)s^3 +  c_1 \frac{s}{n^2}
+\frac{c_2}{n^3} +\frac{3s^2}{n}\left(\tanh(\beta s)-\beta s
\right) ~.
\end{align*}
Taking expectation, we obtain
\begin{align}
 \E S_{t+1}^3 &\leq \left(1+\frac{3\temp}{n}\right)(\E S_t^3) + c_1 \frac{\E S_t}{n^2} +\frac{c_2}{n^3}+\frac{3}{n}\E\left[ S_t^2\left(\tanh(\beta S_t)- \beta S_t \right)\right] ~.\label{eq-3rd-moment-upper-bound}
\end{align}
Now, another application of the FKG inequality, combined with
\eqref{eq-1st-moment-cube-lower-bound} and
\eqref{eq-3rd-moment-upper-bound}, implies that for every
sufficiently large $n$
\begin{align}\label{eq-W-s0-recursion-bound}
  W_{s_0}(t+1) \leq \left(1+\frac{3\temp}{n}\right)W_{s_0}(t)+ \frac{c}{n^2}\E_{s_0} S_t + \frac{c'}{n^{3}}~.
\end{align}
Iterating, while noting that $W_{s_0}(0) = 0$ by definition, we conclude that
\begin{align*}
  W_{s_0}(t) \leq  \sum_{j=1}^{t} \left(1+\frac{3\temp}{n}\right)^{t-j}\left(\frac{c}{n^2} \E_{s_0}S_j + \frac{c'}{n^{3}}\right)~.
\end{align*}
Note that \eqref{eq-St-1st-moment-tanh} implies the following immediate rough upper bound on $\E S_t$:
\begin{equation}\label{eq-e-st-rough-upper-bound}
  \E_{s_{0}} S_{t+1} \leq \left(1+\frac{\temp}{n}\right) \E_{s_0} S_t~.
\end{equation}
Plugging this estimate into \eqref{eq-W-s0-recursion-bound} now gives
\begin{align*}
  W_{s_0}(t) &\leq  \sum_{j=1}^{t} \left(1+\frac{3\temp}{n}\right)^{t-j} \left(\left(1+\frac{\temp}{n}\right)^{j}\frac{c'}{n^2} s_0  + \frac{c}{n^{3}}\right) \\
&  \leq \frac{c}{n^2} \left(s_0+ \frac1n\right) \sum_{j=1}^t \left(1+\frac{\temp}{n}\right)^{3t - 2j}
\\
&\leq \frac{c}{\temp n} \left(s_0+\frac1n\right) \mathrm{e}^{3t \temp /
n}~,
\end{align*}
as required.
\end{proof}
We can now show that, starting from a magnetization of $A/\sqrt{\temp n}$, we have that $S_{T_1^+}$ is at least $n^{-1/4}$ with
probability arbitrarily close to $1$.
\begin{lemma}\label{lem-hit-A-good-n--1/4}
Let $A,\gamma > 0$, and define
$$T_1^*\deq \frac{n}{4\temp}\left(\log (\temp^2 n)- \log (A/2) + (\temp^2 n)^{-1/5}\right)~.$$
Consider the magnetization chain $S_t$ started at $s_0 = A/\sqrt{\temp n}$. Then for some
absolute constant $c>0$, the following holds for any $0 < \ell  < \gamma n/\temp$:
\begin{equation*}%\label{eq-upper-bound-hit-1}
\P_{s_0}(S_{T_1^*+\ell} \leq {n^{-1/4}} ) \leq \frac{c}{A^2}~.
\end{equation*}
\end{lemma}
\begin{proof}
By Lemma \ref{lem-St-3rd-moment-exp-cube}, for every $t \leq T_1^*$ we have
\begin{align*}
W_{s_0}(t) &\leq \frac{c}{\temp n} \left(s_0+\frac1n\right) \left(\temp^2 n\right)^{3/4} \left(A/2\right)^{-3/4} \mathrm{e}^{\frac34 \left(\temp^2 n\right)^{-1/5}} \\
&\leq 2c\cdot \temp\left( s_0 +\frac1n\right) (\temp^2 n)^{-1/4} ~,
\end{align*}
where the last inequality holds for $A \geq 2$ and any sufficiently large $n$.
Thus, substituting the value of $s_0$, for any sufficiently large $n$ we have
\begin{equation}
  \label{eq-W-T1-bound} W_{s_0} (t) \leq \frac{c' A}{n^{3/4}} ~\mbox{ for every }t \leq T_1^*~ .
\end{equation}
We next need a lower bound on $s_t \deq \E_{s_0}S_t$. By \eqref{eq-St-1st-moment-tanh}
and the Taylor expansion of the Hyperbolic tangent \eqref{eq-taylor-tanh}, we have
\begin{equation}\label{eq-st-grow}
\E[S_{t+1}-S_t\mid S_t=s] \geq \frac{1}{n}(\temp s - 2s^3/5 - \tilde{c} s / n)~,
\end{equation}
where the constant $\tilde{c}$ replaced the $O(s/n^2)$ from \eqref{eq-St-1st-moment-tanh}.
Taking expectation and plugging in
\eqref{eq-W-T1-bound}, we obtain that for any $t \leq T_1^*$,
\begin{align}
  s_{t+1} - s_t &\geq \frac{\temp}{n}s_t - \frac{1}{n}\left(\frac25 s_t^3 + c' A n^{-3/4}\right)\nonumber\\
  &\geq \frac{\temp}{n}s_t - \frac{1}{n}(s_t^3 + c' A n^{-3/4})~,\label{eq-s_t-increasing-1}
\end{align}
where in the first inequality the value of $c'$ was increased so that the term $c' A n^{-3/4}$ will absorb the term $\tilde{c}s_t/n^2$. In the second inequality above, we used the fact that $s_t^3 \geq 0$. To see this, first consider the case where $n$ is even. In that case, $\E_0 S_t = 0$ for all $t$ by symmetry, thus the monotone coupling immediately gives that whenever $s_0 \geq 0$ we get $\E_{s_0} S_t \geq 0$ for any $t$. If $n$ is odd, a similar argument achieves this property (coupling with a chain that starts at $\pm\frac1n$ with equal probability).

Observe that \eqref{eq-s_t-increasing-1} implies that $s_t$ increases (it has positive drift) as long as $s_t^2 \leq \temp$. By the assumption that $\temp^2 n\to\infty$, this is guaranteed
whenever $s_t = O(n^{-1/4})$. Now, let
$$b_i=\frac{A a^i}{\sqrt{\temp n}} ~,\quad u_i = \min\{t: s_t > b_i\}~, \mbox{ and } i_1=\min\{i:b_i > 2 n^{-1/4}\}~.$$
Clearly,
$$ i_1 = \bigg\lceil \log_a \frac{2n^{-1/4}}{A/\sqrt{\temp n}}\bigg\rceil \leq \frac{1}{4}\log_a(\temp^2n) - \log_a (A/2)~. $$
The definition of $u_i$, together with the fact that $s_t$ is increasing as long as $s_t = O(n^{-1/4})$, implies that $b_i \leq s_t \leq a b_i$ for any $t\in[u_i,u_{i+1})$. Combined with
\eqref{eq-s_t-increasing-1}, we obtain that
\begin{align*}
 u_{i+1}-u_i &\leq \frac{(a-1)b_i}{n^{-1}\left(\temp b_i - a^3 b_i^3 - c' A n^{-3/4}\right)} =\frac{(a-1)n}{\temp - a^3 b_i^2  -  c'A n^{-3/4} / b_i}\\
 & \leq\frac{n}{\temp} \cdot \frac{a-1}{1 -  a^3(\temp^2 n)^{-1/2} -c'A (\temp^2 n)^{-1/4}}\leq  \frac{n}{\temp} \cdot \frac{a-1}{1 - c'' A(\temp^2 n)^{-1/4}}~,
\end{align*}
where the last inequality holds for any large $n$. Therefore,
\begin{align*}\sum_{i=1}^{i_1} u_{i+1}-u_i &\leq i_1 \frac{n}{\temp} \cdot \frac{a-1}{1 - c''A(\temp^2 n)^{-1/4}} \\
&\leq \frac{n}{\temp} \left(\frac{1}{4} \log(\temp^2 n) - \log (A/2)\right) \cdot \frac{a-1}{\log a}\cdot \frac{1}{1 - c''A (\temp^2 n)^{-1/4}} \\
& \leq \frac{n}{\temp}\left(\frac{1}{4}\log (\temp^2 n) - \log (A/2) + (\temp^2 n)^{-1/5}\right) \quad(~=
T_1^*~) ~,
\end{align*}
where the last inequality follows from a choice of $a=1+n^{-1}$, and holds for a sufficiently large $n$, as the
change in the exponent of $\temp^2 n$ absorbs the logarithmic factor. We conclude that, for a sufficiently large
$n$
\begin{equation}
  \E_{s_0}S_{T_1^*} \geq 2 n^{-1/4}~.
\end{equation}
Now, by Lemma \ref{lem-var-bound-0-n-1/4} we have
\begin{equation}\label{eq-var-st-T_1^*}
\var_{s_0} S_{T_1^*}\leq \frac{c}{A^2 \sqrt{n}}~,
\end{equation}
and hence Chebyshev's inequality gives
\begin{equation*}
\P_{s_0}\left(S_{T_{1}^{*}} \leq \frac32 n^{-1/4}\right) \leq \frac{c}{A^2}~.
\end{equation*}
In order to extend this to $T_1^*+\ell$ for $\ell \in \{0,\ldots,\gamma n/\temp\}$, consider
a second chain $\tS_t$ which we spawn at time $T_1^*$ with an initial value of $\tilde{s}_0 = \frac32 n^{-1/4}$.
By monotone coupling the two chains, it suffices to show that
\begin{equation}\label{eq-tS-lower-bound-n-1/4}\P_{\tilde{s}_0}\left(\tS_\ell \leq n^{-1/4}\right) \leq \frac{c}{A^2}~.
\end{equation}
Recalling the above observation that the series $s_t$ is increasing as long as $s_t = O(n^{-1/4})$, we deduce that
$$ \E_{\tilde{s}_0} \tS_\ell \geq s_0 = \frac32 n^{-1/4}~.$$
On the other hand, by the assumption that $\ell < \gamma n/\temp$, Lemma \ref{lem-var-bound-0-n-1/4} gives that
\begin{equation*}\var_{\tilde{s}_0} \tS_{\ell}\leq \frac{c'(\gamma)}{\temp n}~.
\end{equation*}
Thus, the fact that $\temp^2 n \to \infty$ implies that $\var{\tilde{s}_0} \tS_{\ell} =
o((\E_{\tilde{s}_0} \tS_\ell)^2)$,
hence
$\P_{\tilde{s}_0}(\tS_\ell \leq n^{-1/4}) = o(1)$,
and in particular \eqref{eq-tS-lower-bound-n-1/4} holds. This completes the proof.
\end{proof}
To deduce the lower bound on $\cS_{T_1^+}$ as given in
\eqref{eq-upper-bound-lemma-reach-1}, first observe the following: Given that the magnetization chain $S_t$ and the
censored magnetization chain $\cS_t$ are both started from
the same $s_{0}\geq 0$, the chain $S_t$ is stochastically dominated by $\cS_t$.
Thus, it suffices to prove the given lower bound for $S_{T_1^+}$.

Next, consider $A,\gamma > 0$, and recall that according to Lemma \ref{lem-hit-A-good-state}, the hitting time
to $s_0 = A/\sqrt{\temp n}$ is at most $L = \lfloor \gamma n/\temp \rfloor$ with probability at least $1- c A /\sqrt{\gamma}$.
Lemma \ref{lem-hit-A-good-n--1/4} states that for any $0 \leq \ell \leq L$, the probability that
$S_{T_1^* + L - \ell} \leq n^{-1/4}$ given that $S_0 = s_0$ is at most $c/A^2$. Combining these facts
and summing this probability over $\ell$ gives
$$ \P_0 (S_{T_1^* + L} \leq n^{-1/4}) \leq \frac{c}{A^2} + \frac{cA}{\sqrt{\gamma}}~,$$
and choosing $\gamma = A^3$ and $A$ large implies the required inequality \eqref{eq-upper-bound-lemma-reach-1}.

\subsubsection{Proof of \eqref{eq-lower-bound-lemma-reach-1}: Upper bound of $n^{-1/4}$ for $\cS_{T_1^-}$}
Consider the original magnetization chain $S_t$. Let $s_0 \in \{0,\frac1n\}$.
Recalling the rough upper bound \eqref{eq-e-st-rough-upper-bound} for $\E S_t$:
$$\E_{s_0} S_{T_1^-(\gamma)} \leq s_0 (1 + \frac{\temp}{n})^{T_1^-} = o (n^{-1/4})~.$$
By Lemma  \ref{lem-var-bound-0-n-1/4}, we have
$$\var_{s_0} S_{T_1^-} \leq \frac{ce^{-\gamma}}{\sqrt{n}}~.$$
Therefore, Chebyshev's inequality gives
$$\P_{s_0} (|S_{T_1^-(\gamma)}| \geq n^{-1/4}) \leq c e^{-\gamma}~.$$
Since the distribution of $|S_t|$ and $\cS_t$ are precisely the same, this
completes the proof of \eqref{eq-lower-bound-lemma-reach-1}, and hence
concludes the proof of Theorem \ref{thm-reach-0-n^-1/4}. \qed

\subsection{Getting from $n^{-1/4}$ to $\sqrt{\temp}$}\label{subsec-n-1/4-to-sqrt-temp}
This subsection determines the amount of time it takes $\cS_t$ to reach order  $\sqrt{\temp}$ starting from an initial value of $n^{-1/4}$.
\begin{theorem}\label{thm-reach-n^-1/4-sqrt-temp}
Define
\begin{eqnarray*}
  &T_2 \deq \frac14 (n/\temp)\log (\temp^2 n)~,\\
  &T_2^+ (\gamma) \deq T_2 + \gamma n/\temp\quad,\quad T_2^-(\gamma) \deq T_2 -\gamma n/\temp~.
\end{eqnarray*} The following holds for the censored magnetization chain
$\cS_t$:
\begin{align}
&\lim_{\gamma \to\infty} \liminf_{n\to\infty} \P_{n^{-1/4}}(\cS_{T_2^+ (\gamma)} \geq \mbox{$\frac43$}\sqrt{\temp}) =1~,\label{eq-upper-bound-lemma-reach-2}\\
&\lim_{\gamma \to\infty} \limsup_{n\to\infty}
\P_{n^{-1/4}}(\cS_{T_2^- (\gamma)} \geq \sqrt{\temp})=0~.
\label{eq-lower-bound-lemma-reach-2}
\end{align}
\end{theorem}
\subsubsection{Proof of \eqref{eq-upper-bound-lemma-reach-2}: Lower bound of $\frac43\sqrt{\temp}$ for $\cS_{T_2^+}$}
We will show that for any $\gamma \geq 22$, the magnetization
at time $T_2^+(\gamma)$ will be at least $\frac43\sqrt{\temp}$ with high probability.
Fix some $\gamma \geq 22$ throughout this subsection. By Lemma \ref{lem-St-3rd-moment-exp-cube}, for any $t \leq T_2^+(\gamma)$
\begin{align}
  W_{s_0}(t) &\leq c \mathrm{e}^{3\gamma} \cdot\temp \Big(s_0+\frac1n\Big) \left(\temp^2 n\right)^{-1/4}\nonumber\\
  &\leq c' \cdot \temp s_0 \cdot (\temp^2 n)^{-1/4}~,\label{eq-W-T2}
\end{align}
where $c'=c'(\gamma)$. Defining
$$b_i=\frac{2 a^i}{n^{1/4}} ~,\quad u_i= \min\{t: s_t \geq b_i\} \mbox{ and } i_2=\min\{i:b_i > \sqrt{2\temp}\}~,$$
we get that $i_2 \leq \frac{1}{4}\log_a (\temp^2 n) + \log_a 4$. Combining \eqref{eq-st-grow} and
\eqref{eq-W-T2}, we obtain that
\begin{equation}\label{eq-St+1-increase-up-to-zeta}
s_{t+1} - s_t\geq \frac{\temp}{n}s_t -\frac25 \cdot\frac{s_t^3}{n}- c' \frac{\temp}{n} s_0 (\temp^2 n)^{-1/4}~,\end{equation} where the term $\tilde{c}s/n^2$ from \eqref{eq-st-grow} was absorbed in the last term by increasing $c'$ (and noting that $\temp n^{4/3} \to \infty$, with room to spare).

Again notice that, started from $s_{0} = n^{-1/4}$, by \eqref{eq-St+1-increase-up-to-zeta} we have that $s_t$ is increasing as long as $s_t \leq \sqrt{2\temp}$. We deduce that
\begin{align*}
 u_{i+1}-u_i &\leq \frac{(a-1)b_i}{n^{-1}\left(\temp b_i - 2a^3 b_i^3/5 - c' \temp s_0 (\temp^2 n)^{-1/4}\right)} \\
 &=\frac{(a-1)n}{\temp - 2a^3 b_i^2/5  -  c'' \temp(\temp^2 n)^{-1/4} } \leq
 \frac{n}{\temp} \cdot \frac{a-1}{1 -  2a^3b_i^2/(5\temp) - c' (\temp^2 n)^{-1/4} }\\
 &\leq  \frac{n}{\temp} (a-1)\left(1 + 9 a^3b_i^2/\temp + c'' (\temp^2 n)^{-1/4} \right)~,
\end{align*}
where the last inequality requires that $a < (10/9)^{1/3}$. Since $b_i$ is a geometric series,
\begin{align*}
  \sum_{i=1}^{i_2} u_{i+1}-u_i &\leq i_2 (a-1) \frac{n}{\temp} +
  \frac{9}{a+1}a^3\frac{n}{\temp^2}b_{i_2}^2 + i_2 (a-1) \frac{n}{\temp} c'' (\temp^2 n)^{-1/4} \\
  &\leq \left(\frac{1}{4} \cdot\frac{n}{\temp}\log (\temp^2 n) + 2\frac{n}{\temp}\right) + 19\frac{n}{\temp}+
  \frac{n}{\temp}(\temp^2 n)^{-1/5} \\
  &\leq \frac{1}{4} \cdot\frac{n}{\temp}\log (\temp^2 n) + 22\frac{n}{\temp}\quad(~= T_2^+(22) \leq T_2^+(\gamma)~)~.
\end{align*}
where again we chose $a=1+n^{-1}$, and this holds for large $n$. We conclude that $\E_{n^{-1/4}} S_{T_2^+(\gamma)}
\geq \sqrt{2\temp}$. In order to show concentration, we return to Lemma \ref{lem-var-bound-0-n-1/4}, and get
$$\var_{n^{-1/4}} S_{T_2^+(\gamma)} \leq \frac{c}{\temp n} \left(\sqrt{\temp^2 n} \mathrm{e}^{2\gamma}\right) =
\frac{c\mathrm{e}^{2\gamma}}{\sqrt{n}} = o(\temp)~.$$
Hence, Chebyshev's inequality implies that $S_{T_2^+(\gamma)} \geq \frac43\sqrt{\temp}$ with high probability,
completing the proof of \eqref{eq-upper-bound-lemma-reach-2}.

\subsubsection{Proof of \eqref{eq-lower-bound-lemma-reach-2}: Upper bound of $\sqrt{\temp}$ for $\cS_{T_2^-}$}
This bound will again follow from analyzing the original (non-censored) magnetization chain. We will in fact prove a stronger version of \eqref{eq-lower-bound-lemma-reach-2},
namely that
\begin{equation}\label{eq-stronger-lower-bound-lemma-reach-2} \lim_{n\to\infty}
\P_{n^{-1/4}}(\cS_{T_2^- (\gamma)} \geq \sqrt{\temp})=0 ~\mbox{ for any fixed $\gamma>4$}~.\end{equation}
Fix $\gamma > 4$, and note that the simple bound \eqref{eq-e-st-rough-upper-bound} gives
$$\E_{n^{-1/4}} S_{T_2^-(\gamma)} \leq n^{-1/4} \left(1+ \frac{\temp}{n}\right)^{T_2^-(\gamma)} \leq \mathrm{e}^{-3} \sqrt{\temp}~.$$
Lemma \ref{lem-var-bound-0-n-1/4} gives the following variance
bound
$$\var_{n^{-1/4}} S_{T_2^-(\gamma)} = o(\sqrt{\temp})~.$$
Combining the above bounds on the expectation and the variance, we get
\begin{align*}
\E_{n^{-1/4}} \big|S_{T_2^-(\gamma)}\big| &\leq \sqrt{\E_{n^{-1/4}} \big|S_{T_2^-(\gamma)}\big|^2} \\
&\leq \sqrt{\left(\E_{n^{-1/4}} S_{T_2^-(\gamma)}\right)^2 +
\var_{n^{-1/4}}S_{T_2^-(\gamma)}}\leq \mathrm{e}^{-2} \sqrt{\temp}~.
\end{align*}
Recalling that $\var |X| \leq \var X$ for any random
variable $X$, we immediately get a bound on
$\var_{n^{-1/4}}|S_{T_2^-(\gamma)}|$. From another application of
Chebyshev's inequality, it follows that
$$\P_{n^{-1/4}}(|S_{T_2^-(\gamma)}|\geq \sqrt{\temp}) = o(1)~,$$
completing the proof of \eqref{eq-stronger-lower-bound-lemma-reach-2} and of Theorem
\ref{thm-reach-n^-1/4-sqrt-temp}. \qed

\subsection{Getting from $\sqrt{\temp}$ to $\zeta$}\label{subsec-sqrt-temp-to-zeta}
This subsection, the most delicate one out of the first three subsections, deals with the issue of reaching $\zeta$ from
$\sqrt{\temp}$. Our goal is to establish the following theorem.
\begin{theorem}\label{thm-reach-zeta-sqrt-temp}
Define
\begin{eqnarray*}
  &T_3 \deq \frac{1}{2\left(\zeta^2\frac{\beta}{\temp} - 1\right)} \cdot
\frac{n}{\temp}\log (\temp^2 n)~,\\
  &T_3^+ (\gamma) \deq T_3 +
\gamma n/\temp\quad,\quad T_3^-(\gamma) \deq T_3 -\gamma n/\temp~.
\end{eqnarray*} The following holds for the censored magnetization chain
$\cS_t$ and any $B^* > 0$:
\begin{align}
&\lim_{B\to\infty}\lim_{\gamma \to\infty} \liminf_{n\to\infty} \P_{4\sqrt{\temp}/3}(\cS_{T_3^+ (\gamma)} \geq
\zeta - B/\sqrt{\temp n})=1~, \label{eq-upper-bound-lemma-reach-3}\\
 &\lim_{\gamma \to\infty} \limsup_{n\to\infty}
\P_{4\sqrt{\temp}/3}(\cS_{T_3^- (\gamma)} \geq \zeta -
B^\star / \sqrt{\temp n})=0~.\label{eq-lower-bound-lemma-reach-3}
\end{align}
\end{theorem}
\begin{remark*}
Note that, as $\zeta = \sqrt{\temp} - O(\temp^{3/2})$ we have $$T_3 =
\Big(\frac{1}{4} +
O(\temp)\Big)\frac{n}{\temp}\log(\temp^2n)~.$$
\end{remark*}
\subsubsection{Proof of \eqref{eq-upper-bound-lemma-reach-3}: Lower bound of $\zeta$ for $\cS_{T_3^+}$}\label{sec-T3-upper-bound}
First, we will show that with high probability, the original magnetization chain starting from position $s_0 = \frac43\sqrt{\temp}$ will remain in a certain ``nice''
interval up to time $T_3^+$.
\begin{lemma}\label{lem-tau-3-bound}
Consider the original magnetization chain $S_t$ started from
$s_0 = \frac43\sqrt{\temp}$. Let $\tau_3 = \min\{t : S_t <
\frac{7}{6}\sqrt{\temp}\}$. The
following holds for any fixed $\gamma>0$ and sufficiently large $n$:
\begin{equation}\label{eq-tau3-bound}
  \P_{s_0}(\tau_3 \leq T_3^+(\gamma)) \leq \frac{1}{\temp^2 n}~.
  \end{equation}
\end{lemma}
\begin{proof}
The essence of the argument we use to prove \eqref{eq-tau3-bound} lies in the
following inequality: For any event $A$ and non-negative random variable $Y$,
\begin{equation*}
\P(A) \leq \frac{\E Y}{\E (Y \mid A)}~.
\end{equation*}
The role of $Y$ in the above inequality will be played by the following:
 \begin{align*}
   Y_{\alpha} &\deq \sum_{t< T_3^+(\gamma)} \one\{ S_t < \alpha
   \sqrt{\temp}\}~, \quad
   \hat{Y}_{\alpha} \deq \sum_{t< T_3^+(2\gamma)} \one\{ S_t < \alpha \sqrt{\temp}\}~.
 \end{align*}
 We now wish to bound the probability that $\tau_3 \leq T_3^+(\gamma)$. First, we claim that started at $s_0 = \frac{4}{3}\sqrt{\temp}$,
we have
$$s_t \geq s_0~\mbox{ for any $t \leq T_3^+(\gamma)$}~.$$ By induction, it suffices
to establish that for every such $t$ we have $s_{t+1}\geq s_0$ provided that
$s_t \geq s_0$. To see this, notice that $T_2$ and $T_3$ have nearly the same order,
and hence as long as $t \leq
T_3^+(\gamma)$, by \eqref{eq-St+1-increase-up-to-zeta} we have
$$W_{s_0}(t) = O(\temp s_0 (\temp^2 n)^{-1/4+O(\temp)})~,$$ that is, a bound similar to the one given in \eqref{eq-W-T2}. Hence, $W_{s_0}(t)$ can easily be absorbed into the leading order term of \eqref{eq-st-grow}, giving
$$ s_{t+1} - s_t \geq \frac{1}{n}\left(\frac{3}{4}\temp s_t - \frac{2}{5} (s_t)^3\right)~.$$
Therefore, either $s_t \geq \sqrt{15\temp/8}$ or $s_{t+1} \geq s_t$,
and in any case we get $ s_{t+1} \geq s_0$, as required.

Next, we will bound the probability that $S_t < \frac54 \sqrt{\temp}$ for some $t < T_3$. To this end, we
introduce an intermediate point $\xi$ into our analysis; any arbitrary
$\frac{5}{4} < \xi < \frac{4}{3}$ will do. Plugging
\eqref{eq-var-st-bound} in Chebyshev's inequality gives
\begin{align*}
  \P_{s_0} (S_t \leq \xi \sqrt{\temp}) \leq \frac{c}{\temp^2 n}\left(1+\frac{2\temp}{n}\right)^t~,
\end{align*}
where $c=c(\xi)$, and summing over $t$ gives
\begin{equation}\label{eq-hat-Y-xi-1}
  \E ( \hat{Y}_{\xi}) \leq \frac{c}{\temp^2 n}\cdot \frac{\sqrt{\temp^2 n}
  \mathrm{e}^{\gamma}}{2\temp /n}(\temp^2 n)^{O(\temp)} = \frac{n}{\temp} \cdot \frac{c' \mathrm{e}^{\gamma}}
  {(\temp^2 n)^{1/2 - O(\temp)}}~.
\end{equation}
Furthermore, recall the rough bound \eqref{eq-e-st-rough-upper-bound} which we inferred
from \eqref{eq-St-1st-moment-tanh}. In fact, \eqref{eq-St-1st-moment-tanh} gives that
for $s \geq 2/n$,
$$\E[S_{t+1} \mid S_t =s ] \leq \left(1+\frac{\temp}{n}\right)s ~,$$
that is, before hitting $\xi$, the magnetization chain $S_t$ is a submartingale with a drift bounded from above by $2\temp^{3/2}/n$. Thus,
optional stopping implies
\begin{equation}\label{eq-hat-Y-xi-2}\E\left( \hat{Y}_{\xi} \mid Y_{5/4} > 0 \right) \geq \E_{\frac{5}{4}\sqrt{\temp}} \left[ \min\{ t : S_t \geq \xi\sqrt{\temp}\}\right] \geq cn/\temp~.
\end{equation}
Combining \eqref{eq-hat-Y-xi-1} and \eqref{eq-hat-Y-xi-2}, we deduce that
\begin{equation*}
\P_{\frac{4}{3}\sqrt{\temp}}( Y_{5/4} > 0 ) \leq \frac{c_1 \mathrm{e}^{\gamma}}{(\temp^2 n)^{1/2 - O(\temp)}}~,
\end{equation*}
where $c_1 = c_1(\xi)$. The exact same argument shows that, for some other constants $c_2,c_3$, we have:
\begin{align*}
\P_{\frac{5}{4}\sqrt{\temp}}( Y_{6/5} > 0 ) &\leq \frac{c_2 \mathrm{e}^{\gamma}}{(\temp^2 n)^{1/2 - O(\temp)}}~,\\
\P_{\frac{6}{5}\sqrt{\temp}}( Y_{7/6} > 0 ) &\leq \frac{c_3 \mathrm{e}^{\gamma}}{(\temp^2 n)^{1/2 - O(\temp)}}~.
\end{align*}
Combining the three bounds on $Y_{5/4}$,$Y_{6/5}$ and $Y_{7/6}$, and writing the event $\tau_3 < T_3^+(\gamma)$
conditioned on
 the first time $S_t$ hits below $\frac{5}{4}\sqrt{\temp}$, and similarly below $\frac{6}{5}\sqrt{\temp}$, we conclude that
\begin{equation*}
  \P(\tau_3 < T_3^+(\gamma)) \leq c_1 c_2 c_3 \frac{\mathrm{e}^{3\gamma}}{(\temp^2 n)^{3/2 - O(\temp)}} \leq \frac{1}{\temp^2 n}~,
  \end{equation*}
  where the last inequality holds for any sufficiently large $n$.
  \end{proof}
\begin{remark*} The above method in fact shows that for any constant $m > 0$, we have $\P(\tau_3 < T_3) \leq (\temp^2 n)^m$
  for large enough values of $n$ (one simply has to add extra intermediate points playing similar roles as $5/4$ and $6/5$).
\end{remark*}

Next, we will shift to the censored magnetization chain for a
while. Since $S_t$ will stay within the interval
$(\frac{7}{6}\sqrt{\temp},1)$ with high probability, so will
$\cS_t$. Define $\cZ_t \deq \zeta - \cS_t$, for
the convenience when performing Taylor expansion around $\zeta$. First let us
consider the case where $\cZ_t
> 0$. Recalling that
$$\E[\cS_{t+1} \mid \cS_t=s] \geq s + \frac{1}{n}(\tanh(\beta s)-s) - \frac{c}{n^2}~,$$ consider the Taylor expansion of
the Hyperbolic tangent around $\zeta$,
\begin{align}
\tanh(\beta s) &= \zeta + \beta(1-\zeta^2)(s-\zeta) + \beta^2(-1+\zeta^2) \zeta (s-\zeta)^2\nonumber\\
&+ \frac{\beta^3}{3} (-1+4\zeta^2-\zeta^4)(s-\zeta)^3 + \frac{\tanh^{(4)}(\xi)}{4!}(s-\zeta)^4 ~,\label{eq-tanh-taylor-at-zeta}\end{align}
where $\xi$ is between $\zeta$ and $\beta s$. Adding the fact that $\zeta = \sqrt{\temp} + O(\temp^{3/2})$, we get
\begin{align*}&\E[\cZ_t - \cZ_{t+1} \mid \cZ_t] \geq \frac{1}{n}(\tanh(\beta (\zeta - \cZ_t))-(\zeta - \cZ_t))+\frac{c}{n^2} \\
&= \frac{1}{n}\left(\left(\zeta^2\beta-\temp\right)\cZ_t - \sqrt{3\temp}\cZ_t^2 + \frac{1}{3}\cZ_t^3 +
o\left(\cZ_t^4\right) + O(\temp^{3/2}\cZ_t^2+\temp \cZ_t^3) \right) + \frac{c}{n^2},
\end{align*}
 where the $o(\cZ_t^4)$ term originates from the fact that $\cZ_t = \zeta - \cS_t$, hence the leading order (constant) term in the coefficient of each $\cZ_t$ equals the coefficient of $x^t$ in the Taylor expansion of $\tanh(x)$. This further implies that the term $\frac{1}{3}\cZ_t^3$ can easily absorb all the remaining terms, and we obtain that
\begin{align}\E&[\cZ_t - \cZ_{t+1} \mid \cZ_t] \geq \frac{1}{n}\left(\left(\zeta^2\beta-\temp\right)\cZ_t
- \sqrt{3\temp}\cZ_t^2 + O(\temp \cZ_t^2)\right) +
\frac{c}{n^2}~,\label{eq-Zt-progression-1}
\end{align}
(with room to spare, having increased the error terms for the sake of simplicity). Whenever $\cZ_t < 0$, we need
$\cZ_t$ to approach $0$, hence again the terms $|\cZ_t|$ and $|\cZ_t|^3$ are in our favor, giving
\begin{align}\E&[|\cZ_t| - |\cZ_{t+1}| \mid \cZ_t] \geq \frac{1}{n}\left(\left(\zeta^2\beta-\temp\right)|\cZ_t|
+ \sqrt{3\temp}\cZ_t^2 + O(\temp \cZ_t^2)\right) + \frac{c}{n^2}~.\label{eq-Zt-progression-2}
\end{align}
It is evident from \eqref{eq-Zt-progression-1} and \eqref{eq-Zt-progression-2} that we require a bound on the
second moment of $\cZ_t$. We therefore move on to calculate the variance of the $\cZ_t$-s, which is precisely
the variance of the $\cS_t$-s.
\begin{lemma}\label{lem-var-bound-cens}
Let $\cS_t$ be the censored magnetization chain starting from $s_{0}
= \frac43\sqrt{\temp}$. There exists some constant $c > 0$ so that for any fixed $\gamma > 0$, the following holds provided that $n$ is sufficiently large:
 \begin{align}\label{eq-supercritical-var-st-bound}
\var_{s_0}\cS_t \leq \frac{c}{\temp n}~\mbox{ for any $t \leq T_3^+(\gamma)$}~.
\end{align}
Moreover, the above also holds if the chain is started at any $s_0'\geq s_0$.
\end{lemma}
\begin{proof}
Rearranging \eqref{eq-St-2nd-moment} and combining it with the fact that $|S_t|$ and $\cS_t$ have the same
distribution, we have
\begin{align*}
\E\left[\cS_{t+1}^2 \mid \cS_t = s\right]&\leq s^2 - \frac{\temp
s^2}{5n} - \frac{2}{n}s^2 +
\frac{4}{n^2}+\frac{2s}{n}\left(\tanh(\beta s)
+\frac{\temp}{10}s\right)\\
 &\leq s^2\left(1-\frac{\temp}{5n}\right)
+\frac{2s}{n}\left(\tanh(\beta s)-\left(1-\frac{\temp}{10}\right)
s\right) + \frac{4}{n^2}~.
\end{align*}
Taking expectation, we get another recursion relation for the second moment:
\begin{align*}
 \E &\left[\cS_{t+1}^2\right] \leq \left(1-\frac{\temp}{5n}\right)\E \left[\cS_t^2 \right] + \frac{2}{n}\E\left[\cS_t \Big(\tanh(\beta \cS_t)
-\Big(1-\frac{\temp}{10}\Big) \cS_t\Big)\right] +
\frac{4}{n^2}~.
\end{align*}
Modifying \eqref{eq-St-1st-moment-tanh} in the same spirit, we then obtain another similar recursion relation for the
expectation squared:
%\begin{align*}
%\E &\left [ \cS_{t+1} \mid \cS_t = s \right ] \geq \left(1-\frac{\temp}{10n}\right)s \\
%&+ \frac{1}{2n}\left(\tanh\left(\beta(s+n^{-1})\right)
%+ \tanh\left(\beta(s-n^{-1})\right) - 2\left(1-\frac{\temp}{10}\right) s\right)+\frac{c}{n^2}~,
%\end{align*}
%and
\begin{align}\label{eq-lower-second-moment}
(\E &\left[\cS_{t+1}\right])^2 \geq \left(1 -
\frac{\temp}{5n}\right)(\E \left[\cS_t\right])^2 + \frac{c}{n^2} \nonumber\\
&+\frac{1}{n}\E \left[\cS_t\right]\cdot \E
\left[\left(\tanh(\beta\cS_t) - \left(1-\frac{\temp}{10}\right)
\cS_t\right)\right]  ~.
\end{align}
Define $ \Gamma \deq
\tanh\left(\beta\cS_t\right)-\left(1-\frac{\temp}{10}\right)
\cS_t~, $ according to which we can then rewrite
\eqref{eq-lower-second-moment} as
\begin{equation}\var_{s_0} \left(\cS_{t+1}\right) \leq
\left(1-\frac{\temp}{5n}\right)\var_{s_0} \left(\cS_t\right)
+\frac{\E[\cS_t\Gamma] -\E \cS_t \E\Gamma}{n}+ \frac{c}{n^2}~.
\end{equation}
Crucially, for any $s > \frac{7}{6}\sqrt{\temp}$, the function $
f(s) = \tanh\left(\beta s\right) -\left(1-\frac{\temp}{10}\right)
s $ is decreasing. Hence, conditioning on $\cS_t \geq
\frac{7}{6}\sqrt{\temp}$ we can apply the FKG inequality and
 obtain that
\begin{align*}
  \E[\cS_t\Gamma] &=\P\left(\cS_t \geq \mbox{$\frac{7}{6}\sqrt{\temp}$}\right) \cdot \E\left[\cS_t \Gamma \mid \cS_t \geq \mbox{$\frac{7}{6}\sqrt{\temp}$}\right]  \\
&+ \P\left(\cS_t \leq \mbox{$\frac{7}{6}\sqrt{\temp}$}\right) \cdot \E\left[\cS_t \Gamma \mid \cS_t \leq \mbox{$\frac{7}{6}\sqrt{\temp}$}\right]  \\
&\leq \P\left(\cS_t \geq \mbox{$\frac{7}{6}\sqrt{\temp}$}\right) \cdot \E\left[\cS_t \mid \cS_t \geq
\mbox{$\frac{7}{6}\sqrt{\temp}$} \right] \cdot
  \E \left[\Gamma \mid \cS_t \geq \mbox{$\frac{7}{6}\sqrt{\temp}$}\right] \\
&+ \P(\tau_3\leq t) \cdot \mbox{$\frac{7}{6}\sqrt{\temp}$}\cdot \left(3\temp\cdot
\mbox{$\frac{7}{6}\sqrt{\temp}$}\right)~,
\end{align*}
where in the last inequality we used the fact that the $\Gamma
\leq 3\temp \cS_t$, combined with the condition $\cS_t \leq
\frac{7}{6}\sqrt{\temp}$. Notice that since $\cS_t$ is
non-negative, $$\E[\cS_t \mid \cS_t \geq
\mbox{$\frac{7}{6}\sqrt{\temp}$}] \leq \frac{\E \cS_t}{ \P(\cS_t
\geq \mbox{$\frac{7}{6}\sqrt{\temp}$})}~,$$ and furthermore, $\E
\left[\Gamma \mid \cS_t < \mbox{$\frac{7}{6}\sqrt{\temp}$}\right]
\geq 0~.$ We therefore conclude that
\begin{align*}
\E \cS_t \Gamma &\leq \frac{\E \cS_t \E \Gamma}{\P(\cS_t \geq \mbox{$\frac{7}{6}\sqrt{\temp}$})} + \P(\tau_3\leq t)  \cdot 6\temp^2\\
&= \E \cS_t \E \Gamma\left(1+ \frac{\P(\cS_t < \frac{7}{6}\sqrt{\temp})}{1-\P(\cS_t <
\frac{7}{6}\sqrt{\temp})}\right) + 6\temp^2 \P(\tau_3\leq t)~,
\end{align*}
hence for any $t \leq T_3^+(\gamma)$
\begin{align}
&\E_{s_0} \cS_t \Gamma - \E_{s_0} \cS_t \E_{s_0} \Gamma\leq \P_{s_0}(\tau_3\leq t)\cdot \left(\frac{\E_{s_0}
\cS_t \E_{s_0} \Gamma}{1-\P_{s_0}(\tau_3\leq t)} + 6\temp^2 \right)\nonumber\\
&\leq\frac{1}{\temp^2 n}(4 \temp(\E_{s_0} \cS_t)^2 + 6\temp^2 )\leq\frac{1}{\temp^2 n}(4 (\E_{s_0} \cS_t)^2 +
6\temp^2 )\nonumber \\
&\leq\frac{1}{\temp^2 n}(4 \temp (\E_{s_0} S_t)^2 + 4 \temp \var_{s_0} S_t + 6\temp^2 )\leq\frac{16
\temp^2}{\temp^2 n} = \frac{16}{n}~, \label{eq-bound-expectation-censor}
\end{align}
where in the above inequalities we applied Lemmas \ref{lem-tau-3-bound} and
\ref{lem-var-bound-0-n-1/4}, combined with the facts that the
$\Gamma$ is bounded from above by $3\temp \cS_t$ as well as that $|\E
S_t |\leq \zeta+\sqrt{\temp}$ (since, whenever $|\E S_t| \geq
\zeta$, \eqref{eq-St-1st-moment-tanh} and Jensen's inequality
imply that $|\E S_{t+1}| \leq |\E S_t|$). Altogether,
\begin{align*}
\var_{s_0} \cS_{t+1} &\leq \left(1-\frac{\temp}{5n}\right)\var \cS_t+ \frac{c'}{n^2} ~,
 \end{align*}
and by iterating we get that for any $t\leq T_3^+(\gamma)$
 \begin{align*}
\var_{s_0}\cS_t &\leq \frac{c'}{n^2}\cdot \frac{1}{\temp/5n} =
\frac{c}{\temp n}~.
\end{align*}
To extend the above to any starting position $s_0'\geq s_0$, notice that the only difference is the bound we get on $\tau_3$, which is inferred immediately from
monotone coupling. This completes the proof of the lemma.
\end{proof}

We are now ready to establish the lower bound on $\cS_{T_3^+}$.
\begin{proof}[\textbf{\emph{Proof of \eqref{eq-upper-bound-lemma-reach-3}}}]
At this point, equipped with the variance bound on $\cZ_t$ (the same bound we have for $\cS_t$), we can return to
\eqref{eq-Zt-progression-1} and \eqref{eq-Zt-progression-2}:
\begin{align*}\E\left[|\cZ_t| - |\cZ_{t+1}|\right] &\geq \frac{1}{n}\left(\left(\zeta^2\beta-\temp\right)\E|\cZ_t|
- \sqrt{3\temp}\E \cZ_t^2 +O(\temp \E \cZ_t^2) \right) - \frac{c}{n^2} \\
&\geq \frac{1}{n}\left(\left(\zeta^2\beta-\temp\right)\E|\cZ_t| - \sqrt{3\temp}(\E \cZ_t)^2 +O(\temp (\E
\cZ_t)^2) \right) - \frac{c'}{\sqrt{\temp}n^2}.
\end{align*}
Setting: $$b_i = a^{-i}\left(\zeta-\mbox{$\frac{4}{3}\sqrt{\temp}$}\right)~,i_3 = \min\{i:b_i < 1/\sqrt{\temp
n}\}~,u_i = \min\{t: \E |\cZ_t| < b_i\}~,$$ we get
$$b_i/a \leq \E |\cZ_t| \leq b_i~\mbox{for any $t \in
[u_i,u_{i+1})$}~.$$ since $\E_{s_0}|\cZ_t|$ is decreasing as long as $\E_{s_0}\cZ_t \geq 1/\sqrt{\temp n}$. For sufficiently large $n$, we have
$$i_3 \leq \log_a \frac{\sqrt{3\temp} - \frac{4}{3}\sqrt{\temp} + O(\temp^{3/2})}{1/\sqrt{\temp n}}
\leq \frac{1}{2}\log_a (\temp^2 n)~.$$ Our estimates on $|\cZ_t|-|\cZ_{t+1}|$ yield the following:
\begin{align*}
  u_{i+1}-u_i &\leq \left(\frac{(a-1) b_i}{a}\right)/\Big(
   \frac{\temp}{n}
  \Big( \big(\zeta^2\frac{\beta}{\temp} - 1\big)\frac{b_i}{a} - \sqrt{\frac{3}{\temp}}b_i^2 + c_1 b_i^2\Big)-\frac{c_2}{\sqrt{\temp}n^2}  \Big)\\
  &\leq \frac{n}{\temp\left(\zeta^2\frac{\beta}{\temp} - 1\right)} \cdot \frac{(a-1)} {1 - \frac{\sqrt{3}a}{\zeta^2 \beta/\sqrt{\temp} - \sqrt{\temp}} b_i+c_1 b_i-\frac{c'_2}{\sqrt{\temp^2 n}} }
\end{align*}
where in the last inequality we wrote $$\frac{c_2 a}{\sqrt{\temp}n^2} = \frac{\temp}{n} \cdot \frac{c_2'}{\sqrt{\temp^2 n}} \cdot
\frac{1}{\sqrt{\temp n}} \leq \frac{\temp}{n} \cdot \frac{c_2'}{\sqrt{\temp^2 n}} b_i~.$$
Therefore, as $b_i \leq \zeta - \frac{4}{3}\sqrt{\temp} = (\sqrt{3}-\frac{4}{3})\sqrt{\temp} + O(\temp^{3/2})$, we get
\begin{align*}
 u_{i+1}-u_i &\leq \frac{n}{\temp\left(\zeta^2\frac{\beta}{\temp} - 1\right)} \cdot \frac{a-1} {1 - 2\frac{b_i}{\sqrt{\temp}} -\frac{c'_2}{\sqrt{\temp^2 n}} }~,
\end{align*}
where the last inequality requires that $a < 4/3$. As $2\frac{b_i}{\sqrt{\temp}} + \frac{c'_2}{\sqrt{\temp^2 n}} < \frac{4}{5}$ for a sufficiently large $n$, we conclude that
\begin{align}
  \sum_{i=1}^{i_3} u_{i+1}-u_i &\leq
  \frac{n}{\temp\left(\zeta^2\frac{\beta}{\temp} - 1\right)} (a-1)\sum_{i=1}^{i_3}
   \left(1 + 5\frac{b_i}{\sqrt{\temp}} +5\frac{c'_2}{\sqrt{\temp^2 n}} \right)\nonumber\\
&\leq    (a-1) \frac{n}{\temp} \frac{1}{\left(\zeta^2\frac{\beta}{\temp} - 1\right)} \left(i_3 + \frac{5}{a-1}+\frac{5 c'_2 i_3}{\sqrt{\temp^2 n}}\right)\nonumber\\
&\leq  \frac{1}{2\left(\zeta^2\frac{\beta}{\temp} - 1\right)} \cdot \frac{n}{\temp} \log(\temp^2 n) +
6\frac{n}{\temp}\quad(~= T_3^+(6)~)~.\label{eq-expectation-T3}
\end{align}
where again we chose $a=1+n^{-1}$, and this holds for large $n$.

\begin{remark*}
Since $|\cZ_t|$ is again a supermartingale
with holding probabilities bounded from above, an application of Lemma \ref{lem-supermatingale-positive}
implies that
$$\P_{s_0} \left( \tau_\zeta > T_3^+(\gamma+6)\right) \leq c / \sqrt{\gamma}~,$$
where $\tau_\zeta$ is the hitting time of
$\zeta$, i.e., $\tau_{\zeta} \deq \min\{t: \cS_t \geq \zeta\}$. We thus have:
$$ \lim_{\gamma \to\infty} \limsup_{n\to\infty} \P_{4\sqrt{\temp}/3}(\tau_{\zeta}\geq T_3^+ (\gamma))  =0~.$$
\end{remark*}

Combined with the decreasing property of $\E_{s_0} \cZ_t$ up to $1/\sqrt{\temp n}$, inequality \eqref{eq-expectation-T3}
implies that for any $\gamma \geq 6$ and sufficiently large $n$,
\begin{equation}\label{eq-Z-T-3-bound}
\E_{s_0} |\cZ_{T_3^+(\gamma)}| \leq 1/\sqrt{\temp n}~.
\end{equation}
Together with Lemma
\ref{lem-var-bound-cens} and Chebyshev's inequality, we deduce that
\begin{equation*}
\P_{s_0} (\cZ_{T_3^+(\gamma)} \geq B/ \sqrt{\temp n} ) \leq \frac{c}{(B-1)^2},~
\end{equation*}
for some constant $c$ and hence implies \eqref{eq-upper-bound-lemma-reach-3}.
\end{proof}
\subsubsection{Proof of \eqref{eq-lower-bound-lemma-reach-3}: Upper bound of $\zeta$ for $\cS_{T_3^-}$}
Similar to our definition of $\cZ_t$ for the censored magnetization chain,
 define $Z_t \deq \zeta - S_t$ for the non-censored chain. We first show that $\E Z_t$ is suitably large at $T_3^-$, and then proceed to translate this result to its censored
 analogue $\E\cZ_t$. Since $\tanh^{(2)}(x) \leq 0$, the
Taylor expansion \eqref{eq-tanh-taylor-at-zeta} of $\tanh$ around
$\zeta$ implies that
$$ \tanh(\beta s) \leq \zeta + \beta(1-\zeta^2)(s-\zeta)~.$$
We deduce that
\begin{align*}\E&[Z_t - Z_{t+1} \mid Z_t] \leq \frac{1}{n}(\tanh(\beta (\zeta - Z_t))-(\zeta - Z_t)) \leq \frac{\left(\zeta^2\beta-\temp\right)Z_t}{n} ~,\end{align*}
and therefore
\begin{align}\label{eq-Zt-lower-bound}\E&Z_{t+1} \geq \left(1-\frac{\zeta^2\beta-\temp}{n}\right)\E Z_t ~.\end{align}
Iterating the above inequality and choosing $s_0 = \frac43\sqrt{\temp}$ gives
\begin{equation*}
  \E_{s_0} Z_{T_3^-(\gamma)} \geq \left(1-\frac{\zeta^2\beta-\temp}{n}\right)^{T_3^-(\gamma)} s_0 \geq \frac{\mathrm{e}^\gamma}{\sqrt{\temp^2 n}} \left(\sqrt{3}-\frac{4}{3}\right)\sqrt{\temp} =
  \frac{c' \mathrm{e}^{\gamma}}{\sqrt{\temp n}}~.
\end{equation*}
Note that $Z_{T_3^-(\gamma)} \neq \cZ_{T_3^-(\gamma)}$ only if $S_t = 0$ for some $t < T_3^-(\gamma)$. Noting that, clearly, $\tau_0 > \tau_3$ for the above choice of $s_0$, we thus obtain that
\begin{align*}
\E_{s_0} \cZ_{T_3^-(\gamma)} &\geq
\E_{s_0} Z_{T_3^-(\gamma)} - \sum_{t < T_3^-(\gamma)} \P_{s_0}(\tau_{0} = t) \E_0
\cS_{T_3^-(\gamma) - t}\\
&\geq
\E_{s_0} Z_{T_3^-(\gamma)} - \P_{s_0}(\tau_0 < T_3^-(\gamma))
\max_{t < T_3^-(\gamma)}
\E_0
\cS_{T_3^-(\gamma)-t}\\
&\geq
\E_{s_0} Z_{T_3^-(\gamma)} - \P_{s_0}(\tau_3 < T_3^-(\gamma))
\max_{t < T_3^-(\gamma)}
\sqrt{\var_0 S_{T_3^-(\gamma)-t}}\\
&\geq \E_{s_0} Z_{T_3^-(\gamma)} - \frac{1}{\temp^2 n}\left(\frac{c}{\temp n}\left(1+\frac{2\temp}n\right)^{T_3^-(\gamma)}\right)^{1/2}\\
&\geq \frac{c' \mathrm{e}^{\gamma}}{\sqrt{\temp n}} - o \left(\frac{1}{\sqrt{\temp n}}\right) \geq \frac{c''
\mathrm{e}^{\gamma}}{\sqrt{\temp n}}~,
\end{align*}
where the bound on the variance is by Lemma \ref{lem-var-bound-0-n-1/4}.
Finally, combining Lemma \ref{lem-var-bound-cens} and Chebyshev's inequality, we infer that for some constant
$c_{B^\star}$ depending on $B^\star$,
$$\P_{s_0} (\cZ_{T_3^-(\gamma)} \leq B^\star/\sqrt{\temp n}) \leq c_{B^\star} \mathrm{e}^{-\gamma}~,$$
which then implies \eqref{eq-lower-bound-lemma-reach-3}, and concludes
the proof of Theorem \ref{thm-reach-zeta-sqrt-temp}. \qed

\subsection{Getting from $1$ to $\zeta$}\label{subsec-1-to-zeta}
In this subsection, we consider the problem of reaching $\zeta$ from the other endpoint of the censored
magnetization chain, namely, from 1. The result stated by the following theorem is analogous to Theorem \ref{thm-reach-zeta-sqrt-temp} from the previous subsection.
\begin{theorem}\label{thm-reach-zeta-4}
Define
\begin{eqnarray*}
  &T_4 \deq \frac{1}{2\left(\zeta^2\frac{\beta}{\temp} - 1\right)} \cdot \frac{n}{\temp}\log (\temp^2 n)~,\\
  &T_4^+ (\gamma) \deq T_4 +
\gamma n/\temp\quad,\quad T_4^-(\gamma) \deq T_4 -\gamma n/\temp~.
\end{eqnarray*}
The following holds for the censored magnetization chain
$\cS_t$ and any $B^* > 0$:
\begin{align}
&\lim_{B\to\infty}\lim_{\gamma \to\infty} \liminf_{n\to\infty} \P_{1}(\cS_{T_4^+ (\gamma)} \leq
\zeta + B/ \sqrt{\temp n})=1~, \label{eq-upper-bound-lemma-reach-4}\\
&\lim_{\gamma \to\infty} \limsup_{n\to\infty} \P_{1}(\cS_{T_4^- (\gamma)} \leq \zeta +
B^\star / \sqrt{\temp n})=0~. \label{eq-lower-bound-lemma-reach-4}
\end{align}
\end{theorem}
\subsubsection{Proof of \eqref{eq-upper-bound-lemma-reach-4}: Upper bound of $\zeta$ for $\cS_{T_4^+}$}
Our argument here again hinges on the contraction of the magnetization towards $\zeta$. For convenience, define $\bar{Z}_t \deq (S_t - \zeta)$ to obtain a positive sequence until hitting $\zeta$.
Recalling that by \eqref{eq-St-1st-moment-tanh}, $$\E[S_{t+1} \mid S_t=s] \leq s + \frac{1}{n}(\tanh(\beta s)-s)~\mbox{for $s \geq 0$}~,$$ we can combine Jensen's
inequality with the concavity of the Hyperbolic tangent and get that
\begin{align}\E&[\bar{Z}_{t+1}-\bar{Z}_t] = \E(\E[ S_{t+1}-S_t\mid S_t])
\leq \frac{1}{n}\left(\E \tanh(\beta S_t)- \E S_t\right)\nonumber\\
&\leq \frac{1}{n}\left(\tanh(\beta \E S_t) - \E S_t\right) = \frac{1}{n}\left(\tanh(\beta \E S_t) - \E \bar{Z}_t - \zeta\right)~.
\label{eq-Zt-from-St=1-bound}
\end{align}
Consider the Taylor expansion of $\tanh$ around $\zeta$ as given in \eqref{eq-tanh-taylor-at-zeta}. Since $\tanh^{(4)}(x) < 5$ for any $x \geq 0$, it follows that for a sufficiently large $n$ the term $-\frac{1}{3}(s-\zeta)^3$ absorbs the last term in this Taylor expansion, and hence
\begin{align*}
\tanh(\beta s) \leq \zeta + \beta(1-\zeta^2)(s-\zeta) + \beta^2(-1+\zeta^2) \zeta (s-\zeta)^2.\end{align*}
Therefore, \eqref{eq-Zt-from-St=1-bound} translates into the following
\begin{align}\label{eq-bar-Z-contract}
\E[\bar{Z}_{t+1}-\bar{Z}_t] &\leq \frac{1}{n}\left(-\left(\zeta^2\beta-\temp\right)\E \bar{Z}_t -
\sqrt{3\temp}(\E \bar{Z}_t)^2 + O(\temp (\E \bar{Z}_t)^2)\right)~.
\end{align}
As before, set $$b_i = a^{-i}~,i_4 = \min\{i:b_i < 1/\sqrt{\temp n}\} \mbox{ and }u_i = \min\{t: \E \bar{Z}_t <
b_i\}~.$$ As a bi-product, our analysis of $\E\bar{Z}_t$ will also yield a result on its behavior for the first $n/\temp$ steps, as we will later formulate in Lemma \ref{lem-exp-drop-all-plus} and use in Section \ref{sec:fullmixing}. To this end, we also
define $$i_4'\deq \min\{i:b_i < \sqrt{\temp}\}~.$$ By \eqref{eq-bar-Z-contract} we have that $\E_1
\bar{Z}_t$ is decreasing as long as it is larger than $0$, giving
$$b_i/a \leq \E \bar{Z}_t \leq b_i~\mbox{for any $t \in [u_i,u_{i+1})$}~.$$ It
follows that
\begin{align*}
  u_{i+1}-u_i &\leq \frac{(a-1) b_i / a}
  {n^{-1}
  \left(\left(\zeta^2\beta-\temp\right)\frac{b_i}{a} + \sqrt{3\temp}(\frac{b_i}{a})^2
  - c \temp b_i^2\right)} \\
  &= \frac{(a-1)a^2 n}{\left(\zeta^2\beta-\temp\right)a  + \sqrt{3\temp}b_i - c a^2 \temp b_i}~.
\end{align*}
Hence, absorbing the term $ca^2\temp b_i$ in the term $\sqrt{3\temp}b_i$ gives
\begin{align}\sum_{i=1}^{i_4'} u_{i+1}-u_i &\leq \mathop{\sum_{1\leq i \leq i_0}}_{i: b_i^2 > \temp/a^2} \frac{a^2(a-1)n}{\sqrt{2\temp}b_i}\leq \frac{a^2 n}{\sqrt{2}(1+a)\temp} \leq \frac{n}{\temp}~, \label{eq-side-result-T4}
\end{align}
and
\begin{align*}
\sum_{i=i_4'+1}^{i_4} u_{i+1}-u_i&\leq \mathop{\sum_{1\leq i \leq i_4}}_{i: b_i^2 < \temp} \frac{(a-1)a
n}{\zeta^2\beta-\temp}\\
&\leq \frac{(a-1)n}{\zeta^2\beta-\temp}\left|\left\{ 1\leq i\leq i_0: (\temp n)^{-1} \leq
b_i^2 < \temp\right\}\right|\\
& \leq \frac{1}{2(\zeta^2\frac{\beta}{\temp}-1)}\frac{n}{\temp}\log(\temp^2n)~,
\end{align*}
where we used the fact that the $\{b_i^{-2}\}$ is a geometric series with ratio $a^2$, and in the last
inequality we chose $a=1+n^{-1}$. Adding the last two inequalities together, we get that
\begin{equation}
\sum_{i=1}^{i_4} u_{i+1}-u_i \leq \frac{n}{\temp} +
\frac{1}{2(\zeta^2\frac{\beta}{\temp}-1)}\frac{n}{\temp}\log(\temp^2n) \quad(= T_4^+(1))~.
\label{eq-expectation-T4}
\end{equation}
Thus, for a sufficiently large $n$ and any $\gamma\geq 1$ we have $\E_1 \bar{Z}_{T_4^+(\gamma)}
\leq 1/\sqrt{\temp n}$ (recall the decreasing property of $\E_1 \bar{Z}_t$).
Furthermore, by Lemma \ref{lem-var-bound-cens},
$$\var_{1} \bar{Z}_{T_4^+(\gamma)} = \var_{1} S_{T_4^+(\gamma)} \leq c'/(\temp n)~.$$
Applying Chebyshev's inequality, we therefore deduce that
\begin{equation*}
\P_{s_0} (\bar{Z}_{T_4^+(\gamma)} \geq B/ \sqrt{\temp n} ) \leq \frac{c}{(B-1)^2}
\end{equation*}
for any $\gamma > 1$ and some constant $c$. This completes the proof of \eqref{eq-upper-bound-lemma-reach-4}.

Notice that, in addition, for any $\gamma \geq 1$ and sufficiently large $n$, Cauchy-Schwartz gives the following
\begin{equation}\label{eq-Z-T-4-bound}
\E_1 |\bar{Z}_{T_4^+(\gamma)}| \leq c/\sqrt{\temp n}~.
 \end{equation}
 This bound will be used later on in Section \ref{sec:fullmixing}.

\begin{remark*}
 Recalling that $|\bar{Z}_t|$ is a
supermartingale with holding probabilities bounded uniformly from above, we apply Lemma
\ref{lem-supermatingale-positive} and obtain that for some constant $c$,
$$\P_{1}\left( \tau_\zeta > T_4^+(\gamma+1)\right) \leq c / \sqrt{\gamma}~, $$
where $\tau_\zeta\deq\min\{t: \cS_t -\zeta < n^{-1}\}$ denotes the hitting time to $\zeta$. This immediately implies that
$$\lim_{\gamma \to\infty} \limsup_{n\to\infty} \P_{1}(\tau_{\zeta}\geq T_4^+ (\gamma))  =0~.$$
\end{remark*}

As mentioned above, a bi-product of the above analysis is the following lemma that addresses the behavior of $\bar{Z}_t$ after $n/\temp$ steps.
\begin{lemma}\label{lem-exp-drop-all-plus}
Starting from all-plus configuration, the expected magnetization drops quickly in the first $n/\temp$ steps.
Namely, for sufficiently large $n$ we have $\E_1 \cS_{n/\temp} \leq \zeta + 2 \sqrt{\temp}$.
\end{lemma}
\begin{proof}
To prove the lemma, recall that by \eqref{eq-side-result-T4} and the definition of $i_4'$ we have $$ \E_1 \bar{Z}_{n/\temp} \leq \sqrt{\temp}~.$$
Lemma \ref{lem-var-bound-cens} gives
$$\var_{1} \bar{Z}_{n/\temp} = \var_{1} S_{n/\temp} \leq \frac{c'}{\temp n} \leq \temp~,$$
where the last inequality holds for any sufficiently large $n$. The proof is concluded by Cauchy-Schwartz.
\end{proof}

\subsubsection{Proof of \eqref{eq-upper-bound-lemma-reach-4}: Lower bound of $\zeta$ for $\cS_{T_4^-}$}
We wish to show that a censored magnetization chain started at 1
will satisfy $\cS_{T_4^-} > \zeta + B^\star/\sqrt{\temp n}$ with high probability for any fixed $B^\star > 0$. By the monotone
coupling, it suffices to prove the above statement given any other
starting point. With this in mind, it is convenient to set $s_0 = \zeta
+ \sqrt{\temp}$, and similarly, $z_0 = \sqrt{\temp}$ (continuing the
notation $Z_t = S_t - \zeta$, given in the previous subsection).

As we will show, the magnetization chain has a roughly symmetric behavior in the interval of order $\sqrt{\temp}$ around $\zeta$. In particular, recall that in order to prove Theorem \ref{thm-reach-zeta-sqrt-temp} (that addresses the time it takes $\cS$ to hit $\zeta$ starting from $\frac43 \sqrt{\temp}$, that is, order $\sqrt{\temp}$ away from $\zeta$), we established Lemma \ref{lem-tau-3-bound}, stating that the magnetization stays above $\sqrt{\temp}$ with high probability all along the relevant time-frame.
By following the exact same argument of Lemma \ref{lem-tau-3-bound} it is possible to obtain an analogous symmetric statement: Here the magnetization will always stay below $\zeta + 2 \sqrt{\temp}$ with high probability.
This is formulated in the following lemma. The proof is omitted, as it is essentially identical to that of Lemma \ref{lem-tau-3-bound}.
\begin{lemma}\label{lem-tau-4-bound}
Consider the original magnetization chain $S_t$ started from
$s_0 = \zeta+\sqrt{\temp}$. Let $\tau_4 = \min\{t : S_t >  \zeta + 2\sqrt{\temp}\}$. The
following holds for any fixed $\gamma>0$ and sufficiently large $n$:
\begin{equation*}
  \P_{s_0}(\tau_4 \leq T_4^+(\gamma)) \leq \frac{1}{\temp^2 n}~.
  \end{equation*}
\end{lemma}
Given the above lemma, we can define
 $$\bar{Z}_t^\star \deq \bar{Z}_t \one\{\tau_{4} \geq t\}~,$$
and obtain that the indictor in the
definition of $\bar{Z}^\star$ does not make a real difference.

Using the Taylor expansion of $\tanh(\cdot)$ around $\zeta$ as
given in \eqref{eq-tanh-taylor-at-zeta}, we get
\begin{align*}&\E[\bar{Z}_{t+1} - \bar{Z}_{t} \mid \bar{Z}_t] \geq \frac{1}{n}(\tanh(\beta (\bar{Z}_t + \zeta))-(\bar{Z}_t + \zeta)) - \frac{c}{n^2} \\
&= \frac{1}{n}\left(-\left(\zeta^2\beta-\temp\right)\bar{Z}_t
- \sqrt{3\temp}\bar{Z}_t^2 + \frac{\tanh^{(3)}(\xi)}{6}\bar{Z}_t^3 + O(\temp^{3/2}\bar{Z}_t^2) \right) - \frac{c}{n^2}
\end{align*}
for some $\xi$ between $S_t$ and $\zeta$, and switching to $\bar{Z}_t^\star$ gives
\begin{align*}&\E[\bar{Z}^\star_{t+1} - \bar{Z}^\star_{t} \mid \bar{Z}_t] \geq -\frac{c}{n^2} -2\sqrt{\temp}\P(\tau_4 = t+1 \mid \bar{Z}_t)\\
&+\frac{1}{n}\left(-\left(\zeta^2\beta-\temp\right)\bar{Z}_t^\star
- \sqrt{3\temp}(\bar{Z}_t^\star)^2 - \frac{\tanh^{(3)}(\xi)}{6}(\bar{Z}_t^\star)^3 + O(\temp^{3/2}(\bar{Z}_t^\star)^2)
\right)~.
\end{align*}
Since $\tanh^{(3)}(x) \geq -2$ for any $0 \leq x \leq 2$, and crucially, since $\bar{Z}_t^\star \leq
2\sqrt{\temp}$, changing the coefficient of the term $(\bar{Z}_t^\star)^2$ from $\sqrt{3\temp}$ to
$-3\sqrt{\temp}$ absorbs the entire term $(\bar{Z}_t^\star)^3$ (as well as the error term) for a sufficiently
large $n$. Therefore, using
%\eqref{eq-supercritical-var-st-bound} and
\eqref{eq-Zt-lower-bound} (notice that $Z_t = -\bar{Z}_t$) as well as
Lemma \ref{lem-var-bound-cens}, we have that
\begin{align*}
&  \E \bar{Z}_{t+1}^\star \geq \left(1-\frac{\zeta^2\beta-\temp}{n}\right)\E \bar{Z}_t^\star
- \frac{3\sqrt{\temp}}{n}\E (\bar{Z}_t^\star)^2 -\frac{c}{n^2} -2\sqrt{\temp}\P(\tau_4 = t+1)\\
&\geq \left(1-\frac{\zeta^2\beta-\temp}{n}\right)\E \bar{Z}_t^\star
- \frac{3\sqrt{\temp}}{n} \left(\var \bar{Z}_t + (\E \bar{Z}_t)^2\right) -2\sqrt{\temp}\P(\tau_4 = t+1)~,\\
%&\geq \left(1-\frac{\zeta^2\beta-\temp}{n}\right)\E \bar{Z}_t^\star - \frac{3\sqrt{\temp}}{n}
%\left(\frac{c}{\temp n} + \left(1-\frac{\zeta^2\beta-\temp}{n}\right)^{2t}z_0^2\right)\\
%&-2\sqrt{\temp}\P(\tau_4 = t+1)~.
\end{align*}
and also
$$ \var \bar{Z}_t + (\E \bar{Z}_t)^2 \leq \frac{c}{\temp n} + \left(1-\frac{\zeta^2\beta-\temp}{n}\right)^{2t}z_0^2~.$$
Combining the above two inequalities and iterating, and finally applying Lemma \ref{lem-tau-4-bound}, we obtain the following bound on $\E_{s_0}
\bar{Z}_{T_4^-(\gamma)}^\star$ for sufficiently large $n$:
\begin{align*}\E_{s_0} \bar{Z}_{T_4^-(\gamma)}^\star &\geq \left(1-\frac{\zeta^2\beta-\temp}{n}\right)^{T_4^-(\gamma)}z_0
-\frac{3c}{\sqrt{\temp}n^2}\sum_{t=0}^{T_4^-(\gamma) -1} \left(1-\frac{\zeta^2\beta-\temp}{n}\right)^t\\
&- \frac{3\sqrt{\temp}}{n}\sum_{t=0}^{T_4^-(\gamma) -1}
\left(1-\frac{\zeta^2\beta-\temp}{n}\right)^{2T_4^-(\gamma) -t}z_0^2
-2\sqrt{\temp}\P_{s_0}(\tau_4 \leq T_4^-(\gamma)) \\
&\geq  \frac{\mathrm{e}^{\gamma/2}}{\sqrt{\temp n}} - \frac{c}{\sqrt{\temp}n^2}\cdot
\frac{n}{\zeta^2\beta-\temp} - \frac{3\sqrt{\temp}}{n}\cdot \frac{1}{\sqrt{\temp^2n}}\cdot \frac{n}{\zeta^2\beta-\temp} \cdot \temp -o(\frac{1}{\sqrt{\temp n}}) \\
&\geq \frac{\mathrm{e}^{\gamma/3}}{\sqrt{\temp n}}~,
\end{align*}
where in the last inequality the second and the third term are absorbed in the first term, due to the assumption
$\temp ^2 n\to\infty$. This then implies that
\begin{equation*}
\E_{s_0} \bar{Z}_{T_4^-(\gamma)} \geq \frac{\mathrm{e}^{\gamma/3}}{\sqrt{\temp n}}~.
\end{equation*}
Together with the variance bound of Lemma \ref{lem-var-bound-cens} and Chebyshev's inequality, we obtain that
for any constant $B^{\star} > 0$ there is some $c_{B^\star}$ such that
$$ \P_{s_0} (\bar{Z}_{T_4^-(\gamma)} \leq B^\star/\sqrt{\temp n}) \leq c_{B^\star} \mathrm{e}^{-\gamma/3}~.$$
This concludes the proofs of
\eqref{eq-lower-bound-lemma-reach-4} and Theorem \ref{thm-reach-zeta-4}. \qed

\subsection{Proof of Theorem \ref{thm-cutoff-mag}: magnetization chain cutoff}
Based on the above analysis, we are now ready to establish cutoff for the censored magnetization chain. Define
$$T \deq T_1 +T_2 +T_3~, \quad T^+(\gamma) \deq
T+\gamma n/\temp~ \mbox{ and }\quad T^-(\gamma) \deq T- \gamma n/\temp~.$$
\subsubsection*{Upper bound given worst starting position}
Our goal in this subsection is to prove an upper bound on the cutoff location for $\cS_t$, as specified in Theorem \ref{thm-cutoff-mag}. This bound is an immediate corollary of the following lemma:
\begin{lemma}\label{lem-tau-mag-bound}
Let $\cS_t$ and $\tilde{\cS}_t$ be censored magnetization chains starting from two arbitrary positions $s_0$ and $\tilde{s}_0$, and denote
their coalescence time by $\taumag \deq  \min\{t :
\cS_t = \tilde{\cS}_t\}$. Then there exists a
coupling so that
$$\lim_{\gamma\to\infty}\limsup_{n\to\infty}\P_{s_0, \tilde{s}_0} (\taumag \geq T^+(\gamma))=0~.$$
\end{lemma}
To prove the above lemma, we must first establish that starting from any $s_0$, the censored magnetization
chain at time $T^+(\gamma)$ is fairly close to $\zeta$.
\begin{lemma}\label{lem-cens-mag-close-to-zeta}
Let $\cS_t$ be a censored magnetization chain started from
$s_0\geq 0$. Then $\cS_{T^+(\gamma)}$ will be in an $O\big(\frac{1}{\sqrt{\temp n}}\big)$ interval around $\zeta$ in the following sense:
\begin{equation*}
\lim_{B\to\infty}\lim_{\gamma\to\infty}\limsup_{n\to\infty}\P_{s_0}
(|\cS_{T^+(\gamma)} - \zeta |\geq B/\sqrt{\temp n} )=0 ~.
\end{equation*}
\end{lemma}
\begin{proof}
The proof will follow from the monotone coupling, combined with our results from
the previous subsections. We construct the following couplings of
three chains $\cS^0_t$, $\cS_t$ and $\cS^1_t$, which start from
$0$, $s_0$ and 1 respectively.
\begin{enumerate}
\item At time 0, we start the chains $\cS^0_t$ and $\cS_t$. We construct a monotone coupling of $\cS^0_t$ and $\cS_t$, and run
these two chains up to time $T_1 + T_2$.

\item At time $T_1 + T_2$, the top chain $\cS^1_t$ starting from 1 joins in (for better consistency, we index its time starting from $T_1+T_2$, to match it with the other two chains). Now, we construct a monotone
coupling for these three chains and run them for another $T_3 +
\gamma n/\temp$ steps, with $\gamma$ sufficiently large (note that
$T_3 = T_4$).
\end{enumerate}
By our construction, $\cS^0_t \leq \cS_t \leq \cS^1_t$ holds for
all $t\geq T_1 + T_2$, and in particular at time $t = T_1 + T_2 + T_3
+ \gamma n/\temp$.

 Combining Theorems \ref{thm-reach-0-n^-1/4}, \ref{thm-reach-n^-1/4-sqrt-temp} and \ref{thm-reach-zeta-sqrt-temp} (namely, equations
\eqref{eq-upper-bound-lemma-reach-1},
\eqref{eq-upper-bound-lemma-reach-2} and
\eqref{eq-upper-bound-lemma-reach-3} respectively), we obtain a lower bound for $\cS^0_t$,
and hence for $\cS_t$. On the other hand, Theorem \ref{thm-reach-zeta-4} (namely, equation \eqref{eq-upper-bound-lemma-reach-4}) provides an upper bound for $\cS^1_t$ and hence for $\cS_t$. This concludes the proof of the lemma.
\end{proof}

The above lemma has following immediate corollary, which establishes the concentration of the stationary censored magnetization. To obtain the corollary, simply choose $s_0$ randomly according to the stationary distribution of $\cS_t$ and apply Lemma \ref{lem-cens-mag-close-to-zeta}.
\begin{corollary}\label{cor-cens-mag-stationary}
Denote by $\pi$ the stationary distribution of the censored
magnetization. Then the following holds:
\begin{equation*}
\lim_{B\to\infty}\lim_{n\to\infty}\pi([\zeta - B/ \sqrt{\temp
n}, \zeta + B /\sqrt{\temp n}~]) =1~.
\end{equation*}
\end{corollary}

To continue the proof of Lemma \ref{lem-tau-mag-bound}, we next study the coalescence time of two censored magnetization chains, each starting
from somewhere close to $\zeta$. Recalling that the magnetization is contracting around $\zeta$, we will show that in fact the difference of the above mentioned two magnetization chains behaves essentially like a supermartingale. To be
precise, define
\begin{eqnarray*}&\tau_D \deq \min\Big\{t : |\cS_t- \tilde{\cS}_t| \leq \frac{2}{n}\; \mbox{ or }\;\cS_t\leq \frac76\sqrt{\temp}\; \mbox{ or }\; \tilde{\cS}_t\leq \frac76\sqrt{\temp}\Big\}~,\\
&D_t\deq (\cS_t-\tilde{\cS}_t)\one\{\tau_D \geq t\}~.\end{eqnarray*}
Under these definitions, the following holds:
\begin{lemma}\label{lem-D-t-supermartingale}
Let $\cS_t$ and $\tilde{\cS}_t$ be two censored magnetization
chains started at $s_0$ and $\tilde{s}_0$ resp., with $s_0 \geq \tilde{s}_0 \geq \frac76\sqrt{\temp}$. Then $D_t$ is a supermartingale.
\end{lemma}
\begin{proof}
Noting that there is no difference between censored and non-censored magnetization
for any $t < \tau_D$, the proof below will treat non-censored chains for simplicity.

Note that $D_t =0$ implies that $\tau_D \leq t$ and in particular $D_{t+1} = 0$, therefore the supermartingale condition holds in this case.
It remains to treat the case $D_t > 0$. In this case, by definition we in fact have $D_t \geq 4/n$, which implies that $S_t \geq \frac76\sqrt{\temp}$, that $\tilde{S}_t \geq \frac76\sqrt{\temp}$ and finally that we cannot have $S_{t'} < \tilde{S_{t'}}$
for any $t' \leq t+1$. We now track the slight change in $D_t$ after a single step.
Here and in what follows, let $\mathcal{F}_t$ be the $\sigma$-field generated by these two chains up to time $t$. By definition \eqref{eq-magnet-transit},
\begin{align*}\E[ D_{t+1}-D_t \mid \mathcal{F}_t ] &= \E[ S_{t+1}-S_{t} + \tS_t - \tS_{t+1} \mid| \F_t] \\
&=\frac{1}{n}\left[ f_n(S_t) - f_n(\tS_t)\right]
+\frac{\tS_t-S_t}{n} + \frac{1}{n}\left[ \theta_n(S_t) -
\theta_n(\tS_t)\right]~,
\end{align*}
where
\begin{align*}
  f_n(s)      & \deq \frac{1}{2}\left\{ \tanh[\beta(s+n^{-1})]
                       + \tanh[\beta(s-n^{-1})] \right\} \\
  \theta_n(s) & \deq \frac{-s}{2}\left\{ \tanh[\beta(s+n^{-1})]
                       - \tanh[\beta(s-n^{-1})] \right\}.
\end{align*}
As argued above, $S_t > \tS_t$, hence the Mean Value Theorem implies that
\begin{align*}f_n(S_t) - f_n(\tS_t) = (S_t - \tS_t)\frac{\beta}{\cosh^2(\beta\xi)}~\mbox{ for some $ S_t-n^{-1}\leq\xi \leq S_t+n^{-1}$}~,
\end{align*}
and by the assumption $D_t > 0$ we deduce that $\xi \geq
\frac{7}{6}\sqrt{\temp}$. Recalling that $\cosh(x) \geq
1+\frac{1}{2}x^2$, we get
\begin{align*}
  f_n(S_t) - f_n(\tS_t) &\leq (S_t - \tS_t) \frac{\beta}{(1+\frac{1}{2}(\beta \xi)^2)^2} \leq
  (S_t - \tS_t) \frac{\beta}{1+(\frac{7}{6})^2 \temp}\\
  &\leq \left(1-\frac{\temp}{3}\right)(S_t - \tS_t)~,
\end{align*}
where the last inequality holds for any sufficiently large $n$ (as
$\temp=o(1)$). Applying Taylor expansions on $\tanh$ around $\beta
S_t$ and $\beta \tS_t$, we deduce that
\begin{align*}\theta_n(S_t) - \theta_n(\tS_t) =
 \frac{-S_t}{n\cosh^2(\beta S_t)}  + \frac{\tS_t}{n\cosh^2(\beta \tS_t)} + O(n^{-3})~,
\end{align*}
and since the derivative of the function $x / \cosh^2(\beta x)$ is
bounded by $1$, another application of the Mean Value Theorem
gives
\begin{align*}\left|\theta_n(S_t) - \theta_n(\tS_t) \right| \leq
\frac{S_t-\tS_t}{n} + O(n^{-3})~.
\end{align*}
Altogether,
\begin{align*}\E[ D_{t+1}-D_t \mid \mathcal{F}_t ] &\leq - \frac{\temp}{3n}(S_t - \tS_t)
+ \frac{S_t - \tS_t}{n^2} + O(n^{-4})~,
\end{align*}
hence for a sufficiently large $n$ we obtain that for all $t < \tau_D$,
\begin{equation}\label{eq-D-t-decrease}
\E[D_{t+1}-D_t \mid\F_t]\leq-\frac{\temp}{6n}D_t \leq 0~.
\end{equation}
Altogether, we conclude that $D_t$ is indeed a supermartingale.
\end{proof}

We are now ready to provide an upper bound on the coalescence time of two
chains, each starting from somewhere close to $\zeta$.

\begin{lemma}\label{lem-taumag-start-from-zeta}
There exists some constant $c>0$ so the following holds. Let $B > 0$ and let $\cS_t,\tilde{\cS}_t$ be two censored magnetization chains starting from $s_0, \tilde{s}_0 \in [\zeta - \frac{B}{ \sqrt{\temp n}}, \zeta+ \frac{B}{\sqrt{\temp
n}}]$ resp. Then there exists a coupling of $\cS_t, \tilde{\cS}_t$ with
$$\P_{s_0, \tilde{s}_0} (\taumag \geq B^3 n/\temp) \leq \frac{c}{\sqrt{B}}~.$$
\end{lemma}
\begin{proof}
We run the censored magnetization chains $\cS_t$ and
$\tilde{\cS}_t$ independently until $\tau_D$. Without loss of generality, suppose that $D_0> 0$, and let $W_t \deq \frac{n}{2} D_{t}$. By Lemma
\ref{lem-D-t-supermartingale}, $D_{t} $ is a supermartingale and hence so is $W_t$.

It is easy to verify that $W_t$ satisfies the conditions of
Lemma \ref{lem-supermatingale-positive} with the stopping time
$\tau_D$, by the uniform upper bound for the holding probability of
the magnetization chain and since at most one spin is updated in each step prior to
$\tau_D$ (no censoring comes into effect). Hence, by Lemma
\ref{lem-supermatingale-positive}, together with the bound on $W_0$
due to the assumption $s_0,\tilde{s}_0 \in [\zeta - \frac{B}{ \sqrt{\temp n}}, \zeta+ \frac{B}{\sqrt{\temp n}}]$, we obtain that the following holds for some
constant $c>0$:
\begin{align}\label{eq-tau-0-bound}
\P\left(\tau_D > \frac{B^3 n}{2\temp}\given D_0\right) \leq
\frac{c}{\sqrt{B}}~.
\end{align}

On the event $D_{\tau_D}= 2/n$, we construct a simple monotone coupling
of $\cS_t$ and $\tilde{\cS}_t$, which turns $D_t$ into a
non-negative supermartingale. By
\eqref{eq-D-t-decrease},
$$\E (D_{t+1}- D_t\mid
\mathcal{F}_t) \leq -\frac{\temp}{6 n^{2}}~\mbox{ for $t < \taumag$}~.$$ Therefore, an application of the Optional
Stopping Theorem for non-negative supermartingales gives that for
some constant $c'>0$,
\begin{align}\label{eq-bound-taumag-tau1}
\P\Big(\taumag - \tau_D \geq \frac{B^3 n}{2 \temp}\given D_{\tau_D} =
\frac{2}{n}\Big) \leq \frac{\E (\taumag - \tau_D)}{B^{3}n/2\temp} \leq
\frac{c'}{B}~.
\end{align}
Finally, Lemma \ref{lem-tau-3-bound} implies that for any $t = O (n/\temp)$ we have $D_t = \cS_t
- \tilde{\cS}_t$ with high probability. Altogether, we deduce that there exists a
coupling with the required upper bound on $\taumag$.
\end{proof}
Lemmas \ref{lem-cens-mag-close-to-zeta} and \ref{lem-taumag-start-from-zeta}
immediately complete the proof of Lemma \ref{lem-tau-mag-bound},
which establishes the upper bound for the cutoff in Theorem \ref{thm-cutoff-mag}. \qed

\subsubsection*{Lower bound given worst starting position} In order to establish the lower bound for the cutoff as specified in Theorem \ref{thm-cutoff-mag}, we show that for any fixed $B > 0$,
the censored magnetization starting from 0
satisfies $\cS_{T^-(\gamma)} < \zeta - B/\sqrt{\temp n}$, unlike its stationary distribution.

To see this, we combine Theorems \ref{thm-reach-0-n^-1/4}, \ref{thm-reach-n^-1/4-sqrt-temp} and \ref{thm-reach-zeta-sqrt-temp} (namely, equations
\eqref{eq-lower-bound-lemma-reach-1},
\eqref{eq-lower-bound-lemma-reach-2} and
\eqref{eq-lower-bound-lemma-reach-3}), and deduce that for
any constant $B > 0$
\begin{equation*}
\lim_{\gamma\to\infty}\limsup_{n\to\infty}\P_{0}
(\cS_{T^-(\gamma)} \geq \zeta - B/\sqrt{\temp n})=0~.
\end{equation*}
Together with Corollary \ref{cor-cens-mag-stationary}, it then follows
\begin{equation*}
\lim_{\gamma\to\infty}\liminf_{n\to\infty}
\|P^{T^-(\gamma)}(0,\cdot) - \pi(\cdot)\|_{\mathrm{TV}} =1~,
\end{equation*}
providing the desired lower bound.

\subsubsection*{Cutoff from all-plus starting position}
 The cutoff for the censored magnetization starting from $\cS_0 =1$ will follow from the results we had already proved in order to establish cutoff from the worst starting position.

 Indeed, for the upper bound, we first claim that the following statement holds, analogous to Lemma \ref{lem-cens-mag-close-to-zeta}:
 \begin{equation*}
\lim_{B\to\infty}\lim_{\gamma\to\infty}\limsup_{n\to\infty}\P_1
(|\cS_{T_4^+(\gamma)} - \zeta |\geq B/\sqrt{\temp n} )=0 ~.
\end{equation*}
To see this, construct a monotone coupling of two chains, $\cS_t$ and $\tilde{\cS}_t$, starting from $1$ and $\frac43\sqrt{\temp}$ resp. The above statement then follows from equation \eqref{eq-upper-bound-lemma-reach-3} of Theorem \ref{thm-reach-zeta-sqrt-temp}
and equation \eqref{eq-upper-bound-lemma-reach-4} of Theorem \ref{thm-reach-zeta-4},
together with the fact that $\cS_t \geq \tilde{\cS}_t$ for all $t$.

Therefore, Corollary \ref{cor-cens-mag-stationary} and Lemma \ref{lem-taumag-start-from-zeta} imply that $\cS_t$ will coalesce with the stationary chain at some $t < T_4^+(\gamma)$ with probability arbitrarily close to $1$ (as $\gamma$ increases).

 The lower bound follows from  equation \eqref{eq-lower-bound-lemma-reach-4} of Theorem \ref{thm-reach-zeta-4} combined with Corollary \ref{cor-cens-mag-stationary}, in a manner similar to the proof of the lower bound for the worst starting position.

 This concludes the proof of Theorem \ref{thm-cutoff-mag}. \qed

\section{Cutoff for the entire dynamics} \label{sec:fullmixing}
In this section we prove Theorem \ref{thm-low-temp}. Recalling the definition of $T$,
$T^+(\gamma)$ and $T^+(\gamma)$, we need to show the following:
\begin{align}
&\lim_{\gamma\to\infty}\limsup_{n\to\infty} d_n (T^+(\gamma)) = 0~,\label{eq-upper-bound-full-mixing}\\
&\lim_{\gamma\to\infty}\liminf_{n\to\infty} d_n (T^-(\gamma)) =
1~.\label{eq-lower-bound-full-mixing}
\end{align}

Note that the lower bound for the mixing time of the censored magnetization chain,
as given in Theorem \ref{thm-cutoff-mag}, immediately gives the
desired lower bound \eqref{eq-lower-bound-full-mixing} for the entire dynamics, and it remains to prove \eqref{eq-upper-bound-full-mixing}.

We wish to extend the upper bound we had for the magnetization chain onto the entire dynamics. To this end, we need the following Two Coordinate Chain
Theorem, which was implicitly proved in \cite{LLP}*{Sections 3.3, 3.4} using two-coordinate chain analysis.
Although the authors of \cite{LLP} were considering the case of the original (non-censored)
Glauber dynamics with $0 < \beta < 1$ fixed, one can follow the same arguments and extend that result to censored Glauber dynamics with $\beta = 1+\temp$ where
$\temp = o(1)$. Later on, when we discuss the case of $\temp$ fixed, we shall describe how this argument should be (slightly) modified so that it would hold for any
constant $\beta$.
\begin{theorem}[\cite{LLP}]\label{thm-two-coord-chain}
Let $(\cX_t)$ be an instance of the censored dynamics, $\mu$ the stationary distribution of the dynamics, and suppose $\cX_0$ is supported by
\begin{equation*}\label{eq-Omega-0-def}\Omega_0 \deq \{ \sigma \in \Omega : |S(\sigma)| \leq \mbox{$\frac12$}\}~.\end{equation*}
For any $\sigma_0 \in \Omega_0$ and $\tilde{\sigma} \in \Omega$, we consider the dynamics $(\cX_t)$ starting from $\sigma_0$
and an additional censored dynamics $(\tilde{\cX}_t)$ starting from $\tilde{\sigma}$, and define:
\begin{align*}
  \taumag &\deq  \min\{t : S(\cX_t) = S(\tilde{\cX}_t)\}~,\\
 U(\sigma) &\deq \left|\{i : \sigma(i) = \sigma_0(i) = 1\}\right|~,\quad V(\sigma) \deq \left|\{i : \sigma(i) = \sigma_0(i) = -1\}\right|~,\\
  \Xi &\deq \left\{ \sigma: \min\{ U(\sigma), U(\sigma_0) - U(\sigma), V(\sigma), V(\sigma_0) - V(\sigma))  \} \geq \mbox{$\frac{n}{20}$} \right\}~,\\
  R(t)  &\deq \left| U(\cX_t) - U(\tilde{\cX}_t) \right|~,\\
  H_1(t) &\deq \{ \taumag \leq t \}~,\quad H_2(t_1,t_2) \deq \cap_{i=t_1}^{t_2} \{ \cX_i \in \Xi \;\wedge\; \tilde{\cX}_i \in \Xi \}~.
\end{align*}
For any possible coupling of $\cX_t$ and $\tilde{\cX}_t$, the following
holds for large $n$:
\begin{align}
\max_{\sigma_0 \in \Omega_0} &\| \P_{\sigma_0}(\cX_{r_2} \in \cdot) - \mu \|_{TV}
\leq\mathop{\max_{\sigma_0 \in \Omega_0}}_{\tilde{\sigma}\in\Omega} \Big[\P_{\sigma_0,\tilde{\sigma}}\Big(R(r_1) > \alpha \sqrt{\frac{n}{\temp}}\Big) \nonumber\\
&+ \P_{\sigma_0,\tilde{\sigma}} (\overline{H_1(r_1)}) + \P_{\sigma_0,\tilde{\sigma}} (\overline{H_2(r_1,r_2)}) +
\frac{\alpha c_1 }{\sqrt{r_2-r_1}}\cdot\sqrt{\frac{n}{\temp}} \Big]~,\label{eq-two-coordinate-chain}
\end{align}
and any $r_1 < r_2$ and $\alpha > 0$.
\end{theorem}
We begin with establishing the fact that any instance of the censored Glauber dynamics concentrates on
$\Omega_0$ once it performs an initial burn-in period of $n/\temp$ steps, as incorporated in the following lemma.
\begin{lemma}\label{lem-hit-good-state}
Let $\cX_t$ be the censored Glauber dynamics starting from some starting configuration $\sigma_0$. Then $\cX_{n/\temp} \in \Omega_0$ with high probability.
\end{lemma}
\begin{proof}
By the monotone-coupling of the censored magnetization chain, it suffices to bound $ \P_{1} (|\cS_{n/\temp}|
\geq \frac12)$, i.e., to treat the worst starting state $\sigma_0=1$. Lemma \ref{lem-exp-drop-all-plus} gives that $$\E_1
\cS_{n/\temp} \leq \zeta + 2\sqrt{\temp} = O (\sqrt{\temp})~.$$ Combining with the variance bound given in Lemma
\ref{lem-var-bound-cens} and Cheybeshev's inequality, it follows that
$$\P_{1} (\cS_{n/\temp} \geq \frac{1}{2}) = O \left(\frac{1}{\temp n}\right) = o(1)~,$$
completing the proof.
\end{proof}
\begin{remark*}
  The statement of the above lemma in fact follows directly from the upper bound on $\E_1\cS_{n/\temp}$, without requiring a second moment argument. Nevertheless, we included the above proof as it also holds when $\temp$ is fixed (a case that will be treated in Section \ref{sec:temp-fixed}).
\end{remark*}

 It remains to bound $R(r_1)$ and $H_2(r_1,r_2)$, where the parameters $r_1$ and $r_2$ will be specified later. To do so, we must first
 extend the variance bound given in Lemma \ref{lem-var-bound-cens}
 to the original magnetization chain.
\begin{lemma}\label{lem-var-bound-usual}
Let $S_t$ be a magnetization chain starting from $s_{0} \geq
\frac43\sqrt{\temp}$. Then there exists some constant $c>0$ so that the following holds:
 \begin{align}\label{eq-supercritical-var-st-bound-usual}
\var_{s_0} S_t \leq \frac{c}{\temp n}~,
\end{align}
for any $T_3^+(6) \leq t \leq T_3^+(\gamma)$, any fixed $\gamma$ and any sufficiently large $n$.
\end{lemma}

\begin{proof}
Define $\tau_0 = \min\{t : |S_t| \leq \frac{1}{n}\}$. Recalling
the fact that $|S_t|$ and $\cS_t$ have the same distribution, we
obtain that for any $T_3^+(6) \leq t \leq T_3^+(\gamma)$
\begin{align*}
\var_{s_0}(S_t) &= \var_{s_0} (\cS_t) + (\E_{s_0} \cS_t)^2 -
(\E_{s_0} S_t)^2 \\
& \leq \frac{c}{\temp n} + (\E_{s_0} \cS_t + \E_{s_0} S_t)
(\E_{s_0} \cS_t - \E_{s_0} S_t)\\
& \leq \frac{c}{\temp n} + 2\E_{s_0}\cS_t \cdot \sum_{k=1}^{t}
\P_{s_0} (\tau_{0} =
k) \sqrt{\var_{0}S_{t-k}}\\
&\leq \frac{c_1}{\temp n} + 4 \zeta \cdot \frac{1}{\temp^2 n}
\sqrt{\frac{c_2}{\temp n}\left(1+\frac{2\temp}{n}\right)^t}~,
\end{align*}
where the last inequality follows from \eqref{eq-Z-T-4-bound},
Lemma \ref{lem-tau-3-bound} and Lemma \ref{lem-var-bound-0-n-1/4}.
Note that, as $\temp = o(1)$, we have
$$T_3 = \Big(\frac{1}{4} + o(1)\Big)
\frac{n}{\temp}\log (\temp^2 n) ~\mbox{ and }~ \zeta \leq 4 \sqrt{\temp}~.$$
Altogether, there exists some $c>0$ so that for sufficiently large $n$,
$$\var_{s_0} (S_t) \leq \frac{c}{\temp n}~\mbox{ for any $T_3^+(6) \leq t \leq
T_3^+(\gamma)$}~,$$
as required.
\end{proof}

Now, we are ready to establish an upper bound for the sum of the spins over a prescribed set, as stated by the next lemma.

\begin{lemma}\label{lem-fixed-set-magnetization}
Let $\cX_t$ be the censored Glauber dynamics starting from $\sigma_0$ with corresponding
magnetization $s_0\geq \frac43\sqrt{\temp}$. Then there exists some $c>0$ so the
following holds for any fixed subset $F \subset [n]$, any $\gamma$ and sufficiently large $n$:
%\begin{equation}\label{eq-fixed-set-magnetization}
%\P_{\sigma_0}\Big( \bigcup_{i = T_3^+(6)}^{T_3^+(\gamma^{\star})} \frac{1}{2}\Big|\sum_{i\in G}
%(X_t(i)-\zeta)\Big| \geq n/32 \Big) \leq O(\frac{1}{\temp^2 n})~,\end{equation} for some constant $\gamma^\star$
%and furthermore,
\begin{equation}\label{eq-fixed-set-R}
\E_{\sigma_0} \Big|\sum_{i \in F} (X_{t}(i)-\zeta)\Big| \leq
c\sqrt{\frac{n}{\temp}}~\mbox{for all $T_3^+(6) \leq t\leq T_3^+(\gamma)$}~.
\end{equation}
\end{lemma}
\begin{proof}
Observe that the censored Glauber dynamics $\cX_t$ is
identically distributed as $X_t\cdot \sign(\sum_{i\in [n]}
X_t(i))$. Thus, it is possible to study the censored dynamics via the
original one in the following manner: We construct a monotone coupling
of $X^-_t$, $X_t$ and $X^+_t$, starting from all-minus, $\sigma_0$
and all-plus respectively, such that $X^-_t \leq X_t \leq X^+_t$ for all $t$.
At the same time, we couple $X_t$ and $\cX_t$ so that $\cX_t
=  X_t\cdot \sign(\sum_{i\in [n]} X_t(i))$. Altogether,
\begin{align*}
\sum_{i\in F} \left(\cX_t(i)-\zeta \right) &\leq \max \Big\{\sum_{i\in
F} \left(X^+_t(i) - \zeta \right)\;,\; \sum_{i\in F} \left(- X^-_t(i)
-
\zeta \right)\Big\}\\
&\leq \Big|\sum_{i\in F} (X^+_t (i) - \zeta)\Big| +
\Big|\sum_{i\in F} (X^-_t (i) + \zeta)\Big|~.
\end{align*}

Replacing $F$ with $F^c$ in the above inequality, we obtain

\begin{align*}
\sum_{i\in F^c} \left(\cX_t(i)-\zeta \right)\leq \Big|\sum_{i\in
F^c} (X^+_t (i) - \zeta)\Big| + \Big|\sum_{i\in F^c} (X^-_t (i) +
\zeta)\Big|~,
\end{align*}
which implies that
\begin{align*}
\sum_{i\in F} \left(\cX_t(i)-\zeta \right)\geq n (\cS_t - \zeta) -
\Big(\Big|\sum_{i\in F^c} (X^+_t (i) - \zeta)\Big| +
\Big|\sum_{i\in F^c} (X^-_t (i) + \zeta) \Big|\Big)~.
\end{align*}
Altogether, we have
\begin{align*}
\Big|\sum_{i\in F} \left(\cX_t(i)-\zeta
\right)\Big| &\leq \Big|\sum_{i\in F} (X^+_t (i) -
\zeta)\Big| + \Big|\sum_{i\in F} (X^-_t (i) +
\zeta)\Big| \\
&+\Big|\sum_{i\in F^c} (X^+_t (i) - \zeta)\Big| +
\Big|\sum_{i\in F^c} (X^-_t (i) + \zeta)\Big|
+ n(\cS_t - \zeta)~.
\end{align*}
Squaring and taking expectation, it follows that
\begin{align}\label{eq-partial-sum-pre}
&\frac{1}{5}\E_{\sigma_0}\Big[\sum_{i\in F} \left(\cX_t(i)-\zeta
\right)\Big]^2 \leq \E_{+}\big|\sum_{i\in F} (X^+_t (i) -
\zeta)\big|^2 + \E_{-}\big|\sum_{i\in F} (X^-_t (i) +
\zeta)\big|^2 \nonumber\\
&+ n^2\E_{\sigma_{0}}(\cS_t - \zeta)^2
+\E_{+}\big|\sum_{i\in F^c} (X^+_t (i) - \zeta)\big|^2 +
\E_{-}\big|\sum_{i\in F^c} (X^-_t (i) + \zeta)\big|^2~,
\end{align}
where we absorbed the mixed terms, generated when squaring the former expression, using the multiplying factor of $\frac15$. We now move on to estimating each of the expressions in the right-hand-side of \eqref{eq-partial-sum-pre}.

Combing \eqref{eq-Z-T-4-bound} and Lemma
\ref{lem-var-bound-usual}, we get
$$\E_{\sigma_{0}}(\cS_t
- \zeta)^2 = O \Big(\frac{1}{\temp n}\Big)~.$$ Next, we need to
estimate $\E_+\big|\sum_{i\in F} (X^+_t (i) -
\zeta)\big|^2$. Again by \eqref{eq-Z-T-4-bound}, and also by
symmetry, we infer that $$\Big[\E_{+}\sum_{i\in F} (X^+_t
(i) - \zeta)\Big]^2 = O \Big(\frac{n}{\temp}\Big)~.$$ It remains to
bound the variance for the partial sum:
\begin{itemize}
\item If at time $t$ the spins are positively correlated (by symmetry, the covariances of all the pairs of spins are the
same) then Lemma \ref{lem-var-bound-usual} yields $$\var_{+}\sum_{i\in F}
(X^+_t(i) - \zeta) \leq n^2 \var_{+} S_t =
O\left(\frac{n}{\temp }\right)~.$$

\item If at time $t$ the spins are negatively correlated, then it
follows that $$\var_{+}\sum_{i\in F} (X^+_t(i) - \zeta) \leq
\sum_{i\in F} \var_{+} X^+_t(i) = O (n)~.$$
\end{itemize}
In any case, the variance is $O \big(\frac{n}{\temp}\big)$, and hence
$$\E_{+}\big|\sum_{i\in F} (X^+_t (i) - \zeta)\big|^2 = O
\left(\frac{n}{\temp}\right)~.$$
The remaining three terms in \eqref{eq-partial-sum-pre} are treated similarly (the chains starting from all-plus and all-minus are symmetric). Therefore,
we conclude that for some constant $c > 0$ independent of the choice of $F$,
\begin{align}
\E_{\sigma_0} \Big|\sum_{i \in F} (\cX_t(i)-\zeta)\Big|^2 \leq
\frac{cn}{\temp} ~.\label{eq-second-moment-spins-on-B}
\end{align}
The proof now follows from Cauchy-Schwartz.
\end{proof}

The above lemma will next be used in order to produce upper bounds on $R(r_1)$
and $H(r_1,r_2)$ as defined in Theorem \ref{thm-two-coord-chain}. The next lemma will address the bound on $R(r_1)$, for some $r_1$ to be specified later.
\begin{lemma}\label{lem-bound-R-t}
Consider two instances of the censored Glauber dynamics, $(\cX_t)$ and
$\tilde{\cX}_t$, started at some $\sigma_{0} \in\Omega_0$ and some
arbitrary $\tilde{\sigma}_0$ respectively. Define $R(t)$ and $U(\cX_t)$ as in Theorem \ref{thm-two-coord-chain}. Then there exists some $c>0$ such that for any $\alpha > 0$,
\begin{equation*}%\label{eq-bound-R-t}
\lim_{\gamma\to\infty}\limsup_{n \to \infty}\P_{\sigma_0, \tilde{\sigma}_0} \Big(R(T^+(\gamma)) \geq \alpha
\sqrt{\frac{n}{\temp}}\Big) \leq \frac{c}{\alpha}~.
\end{equation*}
\end{lemma}
\begin{proof}
Let $F = \{i: \sigma_{0} (i) =1\}$ and $E$ be the event
$$E \deq \Big\{\cS_{T_1 +
T_2^+(\gamma /2)} \geq
\mbox{$\frac43$}\sqrt{\temp} ~\wedge~ \tilde{\cS}_{T_1 + T_2^+(\gamma /2)} \geq
\mbox{$\frac43$}\sqrt{\temp}\Big\}~.$$
By definition,
\begin{align*}
|R(t)| &= |U(\cX_t) - U(\tilde{\cX}_t)| = \Big|\sum_{i \in F} \cX_t(i) - \sum_{i\in F} \tilde{\cX}_t(i)\Big|\\
&=\Big|\sum_{i \in F} (\cX_t(i) - \zeta)- \sum_{i\in F}
(\tilde{\cX}_t(i) -\zeta)\Big|\\
&\leq \Big|\sum_{i \in F} (\cX_t(i)
- \zeta)\Big|+ \Big|\sum_{i\in F} (\tilde{\cX}_t(i) -\zeta)\Big|~.
\end{align*}
Together with Lemma \ref{lem-fixed-set-magnetization}, this gives
that
\begin{equation}\label{eq-expectation-bound-R}
\E_{\sigma_0, \tilde{\sigma_0}}\left[ |R(t)| \given E \right] \leq c\sqrt{\frac{n}{\temp}}
\end{equation}
for any $T^+(6+\gamma/2) \leq t\leq T^+(\gamma)$ and sufficiently large $n$.
  Note that
\begin{align*}
&\P_{\sigma_0, \tilde{\sigma}_0} \Big(R(T^+(\gamma)) \geq \alpha \sqrt{\frac{n}{\temp}}\Big)
 \leq \P_{\sigma_0,
\tilde{\sigma}_0} (E^c)+ \P_{\sigma_0, \tilde{\sigma}_0} \Big(R_{T^+(\gamma)}\geq \alpha \sqrt{\frac{n}{\temp}}
\given E\Big)~.
\end{align*}
The first term in the right-hand-side above vanishes as $\gamma \to \infty$ by \eqref{eq-upper-bound-lemma-reach-1} and
\eqref{eq-upper-bound-lemma-reach-2}, and the second term can be bounded by $c / \alpha$ according to
\eqref{eq-expectation-bound-R} and Markov's inequality. This completes the proof.
\end{proof}
We proceed to bound $H_2(r_1, r_2)$, the final ingredient required for applying Theorem \ref{thm-two-coord-chain}.
\begin{lemma}\label{lem-H-2-bound}
Let $\cX_t$ and
$\tilde{\cX}_t$ be two instances of the censored dynamics,
 started at some $\sigma_{0} \in\Omega_0$ and some
arbitrary $\tilde{\sigma}_0$ respectively. Define $H_2(r_1,r_2)$
as in Theorem \ref{thm-two-coord-chain}. The following then holds:
\begin{equation*}%\label{eq-bound-R-t}
\lim_{\gamma_1\to\infty}\lim_{\gamma_2 \to\infty}\limsup_{n \to \infty}\P_{\sigma_0, \tilde{\sigma}_0}
(\overline{H_2 (T^+(\gamma_1), T^+(\gamma_2))}) =0~.
\end{equation*}
\end{lemma}
\begin{proof}
Let $F = \{i : \sigma_{0}(i) =1\}$ and note that $\sigma_{0}\in\Omega_0$ implies that $$\frac{n}4\leq |F| \leq \frac{3n}4~.$$
Next, define:
  \begin{align*}
  E &\deq \Big\{\cS_{T_1 +
T_2^+(\gamma /2)} \geq
\mbox{$\frac43$}\sqrt{\temp} ~\wedge~ \tilde{\cS}_{T_1 + T_2^+(\gamma /2)} \geq
\mbox{$\frac43$}\sqrt{\temp}\Big\}~,\\
    Y& \deq \sum_{T^+(\gamma_1) \leq t \leq T^+(\gamma_2)} \one \Big\{ \Big|\sum_{i \in F}
(\cX_t(i) - \zeta)\Big| > \mbox{$\frac{n}{64}$} \Big\}~.
  \end{align*}
Notice that
\begin{align*}
&\P \bigg(\bigcup_{t=T^+(\gamma_1)}^{T^+(\gamma_2)} \Big\{ \Big|\sum_{i \in F} (\cX_t(i) - \zeta)\Big| \geq \frac{n}{32}
\Big\}\cap E \bigg)\\
 &\leq \P\bigg(\Big\{Y > \frac{n}{128}\Big\}\cap E\bigg) \leq \frac{c_0 \E [Y\one_E]}{n} ~.
\end{align*}
Recall that, \eqref{eq-second-moment-spins-on-B} actually gives that for any choice of $12 < \gamma_1 < \gamma_2$, any $T^+(\gamma_1) \leq t\leq T^+(\gamma_2)$ and any sufficiently large $n$, $$\E \big[|\sum_{i \in F} (\cX_t(i) - \zeta)|^2 \mid E\big]
\leq \frac{cn}{\temp}~.$$ Hence, a straightforward second moment argument gives the following:
\begin{equation}
  \P\left(\Big|\sum_{i \in F}(\cX_t(i) - \zeta)\Big|\one_E > \frac{n}{64}\right) = O\left(\frac{1}{\temp n}\right)~,
\end{equation}
and altogether, $\E_{\sigma_0}[Y\one_E] = O(\temp^{-2})$ and
  \begin{equation*}
\P \bigg(\bigcup_{t=T^+(\gamma_1)}^{T^+(\gamma_2)} \left\{ \Big|\sum_{i \in F} (\cX_t(i) - \zeta)\Big| \ge n/32
\right\}\cap E\bigg) = O\left(\frac{1}{\temp^2 n}\right)~.
  \end{equation*}
An analogous argument for the chain $(\tilde{\cX}_t)$ shows that
     \begin{equation*}
\P \bigg(\bigcup_{t=T^+(\gamma_1)}^{T^+(\gamma_2)} \left\{ \Big|\sum_{i \in F} (\tilde{\cX}_t(i) - \zeta)\Big|
\ge n/32 \right\}\cap E\bigg) = O\left(\frac{1}{\temp^2 n}\right)~.
  \end{equation*}
Combining last two inequalities along with
\eqref{eq-upper-bound-lemma-reach-1} and
\eqref{eq-upper-bound-lemma-reach-2} (that establish that $\P(E)\to 0$ as $\gamma_1\to\infty$) implies the required result.
\end{proof}

Finally, we set $$r_1 = T^+(\gamma)~,~r_2 = T^+(2\gamma)~\mbox{ and }~\alpha = \gamma^{1/4}~.$$ Combining Lemmas
\ref{lem-tau-mag-bound}, \ref{lem-hit-good-state}, \ref{lem-bound-R-t} and \ref{lem-H-2-bound}, then applying Theorem \ref{thm-two-coord-chain} with the above specified parameters, we obtain \eqref{eq-upper-bound-full-mixing},
the required upper bound on the mixing time.

\section{Spectral gap analysis}\label{sec:spectral}
In this section, we prove Theorem
\ref{thm-low-temp-spectral}, which establishes that the spectral gap has order
$\temp /n$.

The following proposition of \cite{DLP} relates the spectral gap
of the original (non-censored) Glauber dynamics for the mean-field Ising model to the spectral gap of its
magnetization chain:
\begin{proposition}[\cite{DLP}*{Proposition 3.9}]
   The Glauber dynamics for the mean-field Ising model and its one-dimensional magnetization chain have the same spectral gap. Furthermore, both gaps are attained by the largest
nontrivial eigenvalue.
\end{proposition}
It was shown in the proof of the above proposition that the spectral gap of the Glauber dynamics is achieved by the second largest eigenvalue. This is also true for the censored Glauber dynamics,
and the proof for the original dynamics extends directly to the
censored one (we omit the full details). Therefore, it remains to estimate the second largest eigenvalue of the censored Glauber dynamics. To do so, as in the case of the non-censored dynamics, we begin by studying the spectral gap of the magnetization chain.

\subsection{Spectral gap of the censored magnetization
chain}
We wish to prove the following result:
\begin{theorem}\label{thm-mag-spectral}
The censored magnetization chain $\cS_t$
satisfies $\gap = \Theta(\temp / n)$.
\end{theorem}
Note that the censored magnetization chain is a
birth-and-death chain on the space $$\magspace \deq \Big\{0,
\frac{2}{n}, \cdots, 1-\frac{2}{n}, 1\Big\}$$ with jumps of
size $\frac{2}{n}$ (for the sake of simplicity, assume that $n$ is
even: For $n$ odd, the only difference is that the initial state $0$ is replaced
with $\frac1{n}$ and all of our arguments remain the same).

For the convenience of notation later on, we define
$$\magspace[a, b]\deq \{x\in \magspace: a\leq x \leq b\}$$
(and similarly, $\magspace (a, b)$, etc., are defined accordingly). We also
introduce the notation $p_x, q_x, h_x$ to denote the transition
probabilities of the chain from $x$ to $x+\frac2{n}$, to $x-\frac2{n}$ and to $x$ respectively, as follows:
\begin{align*}& p_x \deq \cP_\textsf{M}(x, x+\mbox{$\frac{2}{n}$})=\big(2\cdot \one\{x=0\} + \one\{x>0\}\big) \frac{1-x}{2}\cdot\frac{1 - \tanh\beta((x-\frac1n))}{2},\\
&q_x \deq \cP_\textsf{M}(x, x-\mbox{$\frac{2}{n}$}) = \one\{x >0\}\frac{1+x}{2}\cdot\frac{1 + \tanh\beta((x+\frac1n ))}{2},\\
&h_x \deq \cP_\textsf{M}(x, x) = 1- p_x - q_x~,
\end{align*}
where the indicators treat the special case of $x = 0$. By
well known results on birth-and-death chains (see, e.g.,
\cite{LPW}), the conductance $c_x$ of the edge $(x, x+2/n)$, and
the conductance $c'_x$ of the self-loop $(x,x)$ for $x\in
\magspace$ are
\begin{equation*}
c_x \deq \prod_{y\in \magspace(0, x]} \frac{p_y}{q_y}~,\quad c'_x = \frac{h_x}{p_x + q_x} (c_{x-2/n} + c_x)~.
\end{equation*}
We define the total conductance as the sum $$c_S\deq \sum_{x\in\magspace}
(c_x + c_x')~.$$ Finally, for the convenience of using the results of \cite{DLP}, we follow the notation there and define $\xi_i$ as:
$$\xi_1 \deq \sqrt{\frac{1}{\temp n}}~,\quad \xi_2\deq \zeta - \sqrt{\frac{1}{\temp n}}~,\quad\xi_3\deq \zeta+\sqrt{\frac{1}{\temp n}}~.$$

\subsubsection{Lower bound on the spectral gap} The lower bound will rely
on a Cheeger inequality involving the conductance of the chain (not to be confused with the above notion of a conductance of an edge), to be defined next. First, the \emph{edge measure} $Q$, corresponding to a transition kernel $P$, is given by
$$Q(x, y) \deq \pi(x)P(x, y)~,\quad Q(A, B) = \sum_{x\in A, y\in B} Q(x, y)~,$$
and has the following interpretation: $Q(A, B)$ is the probability of moving from $A$ to $B$ in one step when starting from the stationary
distribution. The \emph{bottleneck ratio} of the set $S$ is defined as
$$\Phi(S ) \deq \frac{ Q(S, S^c)}{ \pi(S )}~$$
and the bottleneck ratio of the whole chain is
$$\Phi_\star \deq  \min_{S : \pi(S )\leq 1/2} \Phi(S )~.$$
The beautiful relation between $\Phi_\star$ and the second largest eigenvalue of a chain was established by Alon (1986), Jerrum and Sinclair (1989) and Lawler and Sokal (1988), as formulated by
the following lemma.
\begin{lemma}[\cite{A}, \cite{JS}, \cite{LS}]\label{lem-cheegers-spectral-gap}
 Let $\lambda_2$ be the second largest eigenvalue
of a reversible transition matrix $P$, and $\Phi_\star$ be its corresponding bottleneck ratio. Then
$$\frac{\Phi_\star^2}{2} \leq 1-\lambda_2 \leq 2 \Phi_\star~.$$
\end{lemma}

We therefore proceed to determine the order of $\Phi_\star$ for our censored magnetization chain. The following lemma, together with Lemma \ref{lem-cheegers-spectral-gap}, will immediately provide the desired lower bound of order $\frac{\temp}{n}$ on the spectral gap.

\begin{lemma}\label{lem-bottleneck-lower-bound}
The bottleneck ratio of the censored magnetization chain satisfies $\Phi_\star = \Theta (\sqrt{\temp/n})$.
\end{lemma}
In the following proof and throughout this subsection, we will apply the results from
 the companion paper \cite{DLP} on the conductances of the magnetization chain. Although those results address the original (non-censored) chain, notice that the conductances are the same everywhere except at the origin $0$ (where the corresponding conductances are of the same order).
\begin{proof}
Considering $\zeta$ as the bottleneck, by definition we have
$$\Phi_{\magspace[0,\zeta]} = \frac{\pi(\zeta)p_\zeta}{\pi(\magspace[0,\zeta])} \leq \frac{\pi(\zeta)}{\pi(\magspace[0,\zeta])}\leq \frac{c_{\zeta-\frac2{n}}+c_\zeta+c'_\zeta}{\sum_{x \in \magspace[0, \zeta]}(c_x+c'_x)}~.$$
In the proof of \cite{DLP}*{Lemma 6.2}, it was shown that
\begin{align}
\left\{\begin{array}{ll}
c_x = \Theta(c_\zeta)&\mbox{ uniformly over $x \in \magspace[\xi_2,\xi_3]$}~,\\
c_x = O(c_\zeta)&\mbox{ uniformly over $x \in \magspace$}~,\\
 c'_x = \Theta(c_{x-\frac2{n}} + c_x)~&\mbox{ uniformly over $x \in \magspace$}~.\end{array}\right.\label{eq-cx-cx'-czeta}
 \end{align}
Therefore, we deduce that
$$\Phi_{\magspace[0,\zeta]} \leq \frac{O(c_\zeta)}{\frac{|\zeta-\xi_2|}{2/n}\Theta(c_\zeta)} = O \Big(\sqrt{\frac{\temp}{n}}\Big)~.$$
By symmetry, an analogous argument gives that
$$\Phi_{\magspace[\zeta, 1]} = O
\Big(\sqrt{\frac{\temp}{n}}\Big)~.$$ Altogether, noting that at least one
of $\magspace[0,\zeta]$ and $\magspace[\zeta, 1]$ has stationary
probability no more than $\frac12$, we obtain that
$$\Phi_\star \leq \min \left\{\Phi_{\magspace[0,\zeta]}~,~\Phi_{\magspace[\zeta, 1]}\right\} = O\Big(\sqrt{\frac{\temp}{n}}\Big)~, $$
implying the required upper bound on $\Phi_\star$.

For the lower bound, let $S$ be set minimizing $\Phi_S$ in the definition of $\Phi_\star$. Observe that $S$ is necessarily some interval $\magspace[\xi, \xi']$, by
the structure of the birth-and-death chain. Since we consider
only such sets with $\pi (\magspace[\xi, \xi']) \leq \frac12$, then either
$\magspace[0, \xi]$ or $\magspace[\xi', 1]$ will have stationary
probability at least $\frac14$. Suppose without loss of generality that
$\pi (\magspace[0, \xi]) \geq \pi(\magspace[\xi',1])$. This gives
\begin{equation}\Phi_{S} = \frac{Q(\magspace[\xi, \xi'], \magspace[\xi, \xi']^c)}{\pi(\magspace[\xi, \xi'])} \geq \frac{Q(\magspace[0, \xi],\magspace[0, \xi]^c)}{2\pi(\magspace[0, \xi])} = \frac12 \Phi_{\magspace[0,
\xi]}~,\label{eq-Phi-magspace-lower-bound}\end{equation}
since our assumption implies that $\pi(\magspace[0,\xi])\geq\frac14$.
It therefore remains to show that for some constant $b > 0$ we have $\Phi_{\magspace[0,\xi]} \geq b\sqrt{\temp/n}$.

First, consider the case $\xi \leq \zeta = \xi_2 + \sqrt{1/\temp n}$. In this case we have
\begin{equation}\label{eq-cx-ratio-xi1-xi2}\frac{c_{x+2/n}}{c_x}
\geq 1 + \sqrt{\frac\temp{n}} - O(1/n)~\mbox{ uniformly for $x \in
\magspace[\xi_1, \xi-\sqrt{1/\temp n}]$}~,\end{equation} by \cite{DLP}*{equation
(6.8)}. Therefore, the sum of the $c_x$-s in the above interval is at most the sum
of a geometric series with a quotient of $1/(1+\frac12 \sqrt{\temp/n})$ and initial position $c_\xi$, and it follows that
$$\sum_{x \in \magspace[\xi_1, \xi - \sqrt{1/\temp n}]} c_x \leq 3 \sqrt{\frac{n}{\temp}} \cdot c_\xi~.$$
Furthermore, it follows from \cite{DLP}*{equation (6.4)} that
\begin{equation}
  \label{eq-cx-O(cy)}
   c_x  = O(c_y) \mbox{ uniformly over all $x < y$ in $\magspace[0,\zeta)$}~.
\end{equation}
Altogether, we deduce that
$$\sum_{x \in \magspace[0, \xi]} c_x = O\Big(\sqrt{n/\temp}\Big) c_\xi~.$$
Therefore, noting that $p_x \geq \frac18$ for all $x \leq \zeta$, we conclude that
\begin{equation*}
\Phi_{\magspace[0,\xi]} = \frac{\pi(\xi) p_\xi}{\pi(\magspace[0,\xi])} \geq \frac{c_\xi/8}{\sum_{x\in \magspace[0,\xi]} (c_x+c'_x)}\geq b_1\sqrt{\frac{\temp}{n}}~,
\end{equation*}
where $b_1 > 0$ is some absolute constant.
Together with \eqref{eq-Phi-magspace-lower-bound}, we deduce that in the case $\xi \leq \zeta$ we have
$$ \Phi_S \geq \frac12 \Phi_{\magspace[0,\xi]} \geq \frac12 b_1 \sqrt{\frac{\temp}{n}}~.$$

Second, consider the remaining case where $\xi \geq \zeta$.
By symmetry, a similar argument to the above then shows that in this case, for some
other absolute constant $b_2>0$, we have
$$\Phi_\magspace[\xi, 1] \geq b_2 \sqrt{\frac{\temp}{n}}.$$
Therefore, we immediately have
$$ \Phi_S =
\frac{Q(\magspace[\xi, \xi'], \magspace[\xi, \xi']^c)}{\pi(\magspace[\xi, \xi'])}
\geq
\frac{Q(\magspace[\xi, 1], \magspace[\xi, 1]^c)}{\pi(\magspace[\xi, 1])} \geq b_2 \sqrt{\frac{\temp}{n}}~.$$
Altogether, $\Phi_\star \geq b \sqrt{\temp / n}$ for
$b = \min\{\frac12 b_1,b_2\}$, as required.
\end{proof}

%\begin{remark*}
%Combining Lemma \ref{lem-bottleneck-lower-bound} and Lemma \ref{lem-cheegers-spectral-gap}, we can also deduce
%an upper bound of the spectral gap, of order $\sqrt{n/\temp}$. Unfortunately, this is far from tight. Actually,
%from the form of the inequality in Lemma \ref{lem-cheegers-spectral-gap}, in most cases it is impossible for the
%bottleneck ratio to provide the correct order of upper bound and lower bound at the same time.
%\end{remark*}

\subsubsection{Upper bound on the spectral gap}
Observing that the censored magnetization chain contracts around $\zeta$, our argument for the upper bound on the
spectral gap will be based on the Dirichlet representation, using the test function $\mathds{1} - \zeta$. To this end, we will need
to estimate the fourth moment of $\cS - \zeta$, where $\cS$ is the censored magnetization chain started from the stationary distribution.
\begin{lemma}\label{lem-4th-moment}
The stationary censored magnetization chain satisfies:
$$\E_\pi \left(|\cS - \zeta|^4\right) = O\left((\temp n)^{-2}\right)~.$$
\end{lemma}
\begin{proof}
Using the same notation of Lemma \ref{lem-bottleneck-lower-bound}, let
$$d_x \deq c_x |x - \zeta|^4~,\quad d'_x \deq c'_x |x - \zeta|^4~,$$ and define
$$\xi_2' = \zeta - \frac{32}{\sqrt{\temp n}}~,\quad \xi_3' = \zeta + \frac{32}{
\sqrt{\temp n}}~.$$ We will analyze the decay of $d_x$ as $x$ grows further away from $\magspace[\xi_2', \xi_3']$.
Noting that in \cite{DLP}*{equation (6.10)} it was shown that
$$\frac{c_{x+2/n}}{c_x} \leq 1- \sqrt{\temp / n} +
O(1/n)~\mbox{ for $x \geq \xi_3$}~,$$ we deduce that for $x \geq \xi_3'$ and sufficiently large $n$
$$\frac{d_{x+ 2/n}}{d_x} \leq \frac{c_{x+2/n}}{c_x} \Big(1 + \frac{2/n}{32/ \sqrt{\temp n} }\Big)^4 \leq 1- \frac{1}{2}\sqrt{\frac{\temp}{n}}~,$$
which implies that \begin{equation}
  \label{eq-dx-xi3-1}
  \sum_{x \in\magspace[\xi_3', 1]}d_x \leq 4 \sqrt{\frac{n}{\temp}} \cdot d_{\xi_3'}~.
\end{equation}
Similarly, an analogous argument using \eqref{eq-cx-ratio-xi1-xi2} gives
that \begin{equation}
  \label{eq-dx-xi1-xi2}
\sum_{x\in\magspace[\xi_1, \xi_2']} d_x = O \Big(\sqrt{\frac{n}{\temp}} \cdot d_{\xi_2'}\Big)~,\end{equation}
Now, recall that $\xi_1 = O(1/\sqrt{\temp n}) = o(\zeta)$, which together with \eqref{eq-cx-O(cy)} yields
$$
d_x = O(d_{\xi_1}) ~\mbox{ uniformly over $x \in \magspace[0,\xi_1]$}~,
$$
and since $d_{\xi_1} = O(d_{\xi'_2})$ (again by \eqref{eq-cx-ratio-xi1-xi2}), we get
\begin{equation}
  \label{eq-dx-0-xi1}
\sum_{x\in\magspace[0,\xi_1]} d_x = O \Big(\sqrt{\frac{n}{\temp}} \cdot d_{\xi'_2}\Big)~.
\end{equation}
Finally, in the interval $\magspace[\xi'_2,\xi'_3]$ by \eqref{eq-cx-cx'-czeta} we have
$$ d_x = O\Big(\frac{1}{(\temp n)^2} \cdot c_\zeta \Big)~\mbox{ uniformly over $x \in \magspace[\xi'_2,\xi'_3]$}~,$$
and therefore
\begin{equation}
  \label{eq-dx-xi2-xi3}
\sum_{x\in\magspace[\xi'_2,\xi'_3]} d_x = O \Big(\sqrt{\frac{n}{\temp}} \frac{1}{(\temp n)^2} \cdot c_\zeta \Big)~.
\end{equation}
Combining \eqref{eq-dx-0-xi1}, \eqref{eq-dx-xi1-xi2}, \eqref{eq-dx-xi2-xi3} and \eqref{eq-dx-xi3-1}, we conclude that
$$
\sum_{x\in\magspace} d_x = O \Big(\sqrt{\frac{n}{\temp}} \frac{1}{(\temp n)^2} \cdot c_\zeta \Big)~.
$$
As \eqref{eq-cx-cx'-czeta} gives that $d'_x = O(d_x)$ uniformly over $x\in\magspace$, we further have that
$$
\sum_{x\in\magspace} (d_x + d'_x) = O \Big(\sqrt{\frac{n}{\temp}} \frac{1}{(\temp n)^2} \cdot c_\zeta \Big)~.
$$
Now, by \cite{DLP}*{Lemma 6.2} we have
\begin{equation}
  \label{eq-cs-order}
 c_S = \Theta\left(\sqrt{\frac{n}{\temp}}\cdot c_\zeta\right)~,
\end{equation} and altogether
$$\E_\pi (|\cS - \zeta|^4) = \frac{\sum_{x\in\magspace} (d_x+d'_x)}{c_S} =O\left((\temp n)^{-2}\right)~,$$
as required.
\end{proof}
\begin{remark*}
Using the above method, one can obtain that for
any fixed $k$ we have $\E_\pi |\cS - \zeta|^k = O
\left((\temp n)^{-k/2}\right)$.
\end{remark*}
Another ingredient required for the upper bound on the gap is the next estimate on $\pi(0)$, which is readily obtained from our previous results on the conductances of this chain.
\begin{lemma}\label{lem-pi-0}
The stationary distribution of the censored magnetization chain satisfies
$\pi(0) = O(1/\zeta n)$.
\end{lemma}
\begin{proof}
Following the notation of the previous lemmas, recall that (as stated before), \cite{DLP}*{equation
(6.8)} gave that
\begin{equation*}\frac{c_{x+2/n}}{c_x}
\geq 1 + \sqrt{\frac\temp{n}} - O(1/n)~\mbox{ uniformly for $x \in
\magspace[\xi_1, \xi_2]$}~.\end{equation*}
In particular, we have $$c_S \geq \frac{|\xi_2 - \xi_1|}{2/n}c_{\xi_1} = \frac{\zeta - 2/\sqrt{\temp n}}{2/n} c_{\xi_1} \geq \frac14 \zeta \cdot n \cdot c_{\xi_1} ~,$$
where the last inequality holds for large $n$, as $\temp^2 n\to\infty$ with $n$.
In addition, \eqref{eq-cx-cx'-czeta} and \eqref{eq-cx-O(cy)} imply that
 $$ c_0 = O(c_{\xi_1})~,~c'_0 = O(c_{\xi_1})~.$$ Altogether, we have
$$ \pi(0) =  \frac{c_0+c'_0}{c_S} = O\Big(\frac{1}{\zeta n}\Big)~,$$
as required.
\end{proof}
We conclude the proof of the upper bound on the spectral gap of $\cS_t$ with
the following simple lemma, which provides a lower bound on the $\var_\pi \cS_t$.
\begin{lemma}\label{lem-var-lower-bound}
There exists a constant $b > 0$ so that the stationary censored magnetization chain satisfies $\var_{\pi}\cS_t \geq b / (\temp n)$.
\end{lemma}
\begin{proof}
By \eqref{eq-cx-cx'-czeta} and \eqref{eq-cs-order} we have
$$ c_x = \Theta(c_\zeta) = \Theta\Big(\sqrt{\frac{\temp}{n}} \cdot c_S\Big)~
\mbox{ uniformly over $x \in \magspace[\xi_2,\xi_3]$}~.$$
It follows that there exists some constant $b' > 0$ such that $\pi(x) \geq b'\sqrt{\temp/n}$ for all $x \in \magspace[\xi_2,\xi_3]$ and every $n$. As the
interval $\magspace[\xi_2,\xi_3]$ consists of $\sqrt{n/\temp}$ elements, the required
results immediately follows.
\end{proof}

Now, we are ready to establish the upper bound for the spectral gap. Applying Dirichlet's representation of
the spectral gap using the test function $f = \mathds{1} - \zeta$, we obtain that
\begin{equation}
  \label{eq-dirichlet-form}
  \gap \leq \frac{\E_\pi [(\cS_t - \E (\cS_{t+1}\mid \cS_t))(\cS_t - \zeta) ]}{\var_\pi \cS_t}~.
\end{equation}
Recalling \eqref{eq-St-1st-moment-tanh}, we have
\begin{align*}
 \E[\cS_{t+1} \mid \cS_t=s\;,\;s > 0] &= \E[S_{t+1} \mid S_t=s\;,\; s>0]\\
 &= s + \frac{1}{n}(\tanh(\beta s)-s) + O(1/n^2)~,\\
 \E[\cS_{t+1} \mid \cS_t=0] &= O(1/n)~,
\end{align*}
and therefore, using the Taylor
expansion \eqref{eq-tanh-taylor-at-zeta} of $\tanh$ around $\zeta$, we deduce that
\begin{align*}
&\E_\pi \left[(\cS_t - \E (\cS_{t+1}\mid \cS_t))(\cS_t - \zeta) \right]
\leq \pi(0)\zeta \cdot O(1/n) \\
&+ \frac{1}{n}\left[(\beta \zeta^2 - \temp) \E_\pi |\cS_t - \zeta|^2 + \beta^2 \zeta \E_\pi|\cS_t - \zeta|^3
+ O \left(\E_\pi|\cS_t-\zeta|^4 + \frac{1}{n^2}\right)\right]\\
&= O (1/n^2)~,
\end{align*}
where in the last inequality we plugged in Lemmas \ref{lem-4th-moment} (in order to bound the 2nd, 3rd and 4th moments) and \ref{lem-pi-0} (the upper bound on $\pi(0)$). Plugging this in the Dirichlet form \eqref{eq-dirichlet-form}, and using the variance bound given in Lemma \ref{lem-var-lower-bound}, we obtain that $\gap = O(\temp / n)$, as required.

This concludes the proof of Theorem \ref{thm-mag-spectral}. \qed
\subsection{Spectral gap of the censored Glauber dynamics}

It is easy to verify that every eigenvalue of the censored magnetization chain
is also an eigenvalue of the entire dynamics (via the natural projection of each configuration onto its magnetization). Thus, the upper bound for the spectral gap of $\cS_t$, given in the previous subsection, immediately yields the desired
upper bound for the gap of $\cX_t$. It remains to provide a matching lower bound.

Define
\begin{align*}
\Omega_k &\deq \left\{\sigma: \mbox{$\sum_{i}$} \sigma_i =k\right\}~,\\
 F &\deq \left\{f
: \Omega \mapsto \R\right\}~,\\
F_1 &\deq \left\{f \in F :~\mbox{ For all $k$, }
f \mbox{ is constant over } \Omega_k\right\}~,\\
F_2 &\deq \left\{f \in F :~\mbox{ For all $k$, } \mbox{$\sum_{\sigma\in\Omega_k}$} f(k)
= 0\right\}~.
\end{align*}
Clearly, $F
=F_1\oplus F_2$, and the transition kernel $\cP_\textsf{M}$ preserves the two spaces $F_1$ and $F_2$. Moreover, the lower bound for the spectral gap of $\cS_t$, as stated in Theorem \ref{thm-mag-spectral}, implies that there exists
some universal $b > 0$, so that for any
non-constant eigenfunction $f\in F_1$ corresponding to some eigenvalue $\lambda$, $$\lambda \leq 1 -
b\cdot \temp/n~.$$ Next, we need to treat
the eigenfunctions in $F_2$. We need the following straightforward lemma, proved implicitly in \cite{LLP} for the original (non-censored) dynamics. Its proof
extends directly to our setting of the censored Glauber dynamics:
\begin{lemma}[\cite{LLP}*{Section 2.6})]\label{lem-contraction-same-mag}
Let $\dist(\cdot)$ be the Hamming distance on $\Omega$, and consider two instances of the censored Glauber dynamics,
 $\cX_t,\tilde{\cX}_t$ starting from $\sigma,\tilde{\sigma}\in\Omega$ resp.,
 such that $\cS_0 = \tilde{\cS}_0$.
 Then for any $\beta>0$ there exists a coupling of $\cX_t$ and $\tilde{\cX_t}$
 such that $\cS_1 = \tilde{\cS}_1$ and for some
constant $c > 0$,
$$\E_{\sigma, \tilde{\sigma}}~ [\dist(\sigma_1, \tilde{\sigma}_1)] \leq (1-  c/n) \dist(\sigma, \tilde{\sigma})~.$$
\end{lemma}

In order to translate the above contraction property of the dynamics into an eigenvalue bound, we follow the
ideas of Chen \cite{Chen} (see also \cite{LPW}*{Theorem 13.1}).

\begin{lemma}\label{lem-lower-spectral-contraction}
There exists some constant $c > 0$, so that every eigenvalue $\lambda$ of the censored Glauber dynamics with corresponding eigenfunction $f\in F_2$ satisfies $1- \lambda \geq c/n$.
\end{lemma}
\begin{proof}
Define the \emph{varied Lipschitz constant} of a function $f$
on the space $(\Omega, \dist)$ as
$$\lip (f) \deq \mathop{\max_{\sigma, \tilde{\sigma}\in \Omega_k}}_{0\leq k \leq n} \frac{|f(\sigma)- f(\tilde{\sigma})|}{\dist(\sigma, \tilde{\sigma})}~.$$
Using the coupling in Lemma \ref{lem-contraction-same-mag}, we
infer that for any $k$ and $\sigma, \tilde{\sigma} \in \Omega_k$,
\begin{align*}
|\cP f (\sigma) - \cP f (\tilde{\sigma})| &=
|\E_{\sigma,\tilde{\sigma}}[ f(\sigma_1)-f(\tilde{\sigma_1})]| \leq
\E_{\sigma,\tilde{\sigma}}|f(\sigma_1)-f(\tilde{\sigma_1})|\\
&\leq \lip(f) \cdot \E_{\sigma, \tilde{\sigma}} [\dist(\sigma_1,
\tilde{\sigma}_1)]\leq (1-c/n)\lip(f) \cdot \dist(\sigma, \tilde{\sigma})~,
\end{align*}
where in the last two inequalities we use the definition of the
varied Lipschitz constant and applied Lemma
\ref{lem-contraction-same-mag}. This proves that
$$\lip (\cP f) \leq (1-c/n) \lip(f)~,$$
which then completes the proof of the lemma, by noting that $\lip(f) > 0$
whenever $0\not\equiv f \in F_2$.
\end{proof}

This establishes the order of the spectral gap of $\cX_t$, thus completing
the proof of Theorem \ref{thm-low-temp-spectral}. \qed

%\begin{remark*}
%We have shown in \cite{DLP}*{Lemma 3.5}, that for the Glauber
%dynamics, the spectral gap of the magnetization chain and of the
%dynamics are actually the same. That proof does not apply for the
%censored dynamics because the censored dynamics is not monotone
%and we cannot infer that there exists an increasing eigenfunction
%corresponding to the second largest eigenvalue.
%\end{remark*}
\begin{remark*}
In the special case $\temp = o(1)$, the arguments in the section in fact imply that
the censored magnetization chain $\cS_t$ and the censored Glauber dynamics $\cX_t$
have precisely the same spectral gap (as opposed to simply having the same order).
\end{remark*}

\section{The case of fixed low temperature}\label{sec:temp-fixed}
Thus far, we proved Theorem \ref{thm-low-temp-spectral}
for any $\temp$ with $\temp^2 n \to\infty$, and established
Theorem \ref{thm-low-temp} for the special case
of such $\temp$ with $\temp = o(1)$. In this section, we extend the
statement of Theorem \ref{thm-low-temp} to the case of $\temp > 0$ fixed.

We note that the arguments used in the proof of Theorem \ref{thm-low-temp} hold
almost without change for this case of $\temp > 0$, and our only reason for
distinguishing between these two cases was to simplify some of the statements
and formulas (since whenever $\temp = o(1)$ we have $\zeta = (1+o(1))\sqrt{3\temp}$, rather than
some fixed constant). In fact, several of the complications in the case $\temp = o(1)$
disappear when $\temp$ is fixed, such as our arguments which carefully tracked down the precise power of $\temp$
in various settings.

Therefore, in what follows we list the required modifications that one needs to make
in order to extend the proof in Sections \ref{sec:mag} and \ref{sec:fullmixing}
to the considerably simpler case of $\temp$ fixed.

\subsubsection*{Analysis of hitting a magnetization of $\zeta$ starting from $0$:}
In Section \ref{sec:mag} we introduced the intermediate points $n^{-1/4}$ and $\sqrt{\temp}$
in order to estimate the time it takes $\cS_t$ to hit $\zeta$ starting from $0$
(see Theorems \ref{thm-reach-n^-1/4-sqrt-temp} and \ref{thm-reach-zeta-sqrt-temp}).
Since now we could have $\temp$ large enough so that $\sqrt{\temp} > \zeta$, one
needs to modify the above mentioned second intermediate point, replacing
$\sqrt{\temp}$ by, say, $\zeta / 2$. This includes adjusting the finer level
of intermediate points chosen in Subsection \ref{subsec-sqrt-temp-to-zeta},
i.e., $\frac76\sqrt{\temp}$ should be replaced by $\zeta/3$ and so on.

\subsubsection*{Estimates of Hyperbolic tangent:} Throughout Sections \ref{sec:mag}
and \ref{sec:fullmixing}, we apply a Taylor expansion to analyze the change
in the magnetization (see \eqref{eq-taylor-tanh} and \eqref{eq-tanh-taylor-at-zeta}).
For simplicity, we used the fact that $\temp = o(1)$ when estimating the error terms
in these formulas, and note that a straightforward application of the Mean Value Theorem
gives the required bounds in the case of $\temp$ fixed.

\subsubsection*{Bound on $\tau_3$: ``escaping'' from around $\zeta$:}
In Lemma \ref{lem-tau-3-bound} we study the probability of $\cS_t$ dropping below
$\frac76 \sqrt{\temp}$ (defining $\tau_3$ to be this corresponding hitting time)
given an initial position of $\frac43 \sqrt{\temp}$. Following the above mentioned
modification to the intermediate position $\sqrt{\temp}$, we should now define
$\tau_3$ as the hitting time to $\zeta/3$, and the new statement of Lemma \ref{lem-tau-3-bound}
would be that
$$\P_{\zeta/2} (\tau_3 \leq T^+_3(\gamma))
\leq \frac{1}{n}~.$$
In our original proof of Lemma \ref{lem-tau-3-bound}, we used that the term $(\temp^2 n)^{1/2-O(\temp)}$ is roughly
$(\temp^2 n)^{1/2}$, as $\temp = o(1)$. Whenever $\temp > 0$ is fixed, we simply
reapply the intermediate points analysis with additional points
$$ \zeta/3= \xi_0 < \xi_1 < \ldots < \xi_K = \zeta/2~,$$
where $K=K(\temp)$ is some sufficiently large constant. The rest of the proof
of Lemma \ref{lem-tau-3-bound} holds without requiring any changes.

\subsubsection*{Two coordinate chain analysis:}
The Two Coordinate Chain Theorem formulated in \cite{LLP} was
bdesigned for the $\beta < 1$ case, where the stationary magnetization concentrates around 0. For the
case $\beta = 1 +o(1)$, the stationary magnetization concentrates around $\pm \zeta$ instead,
and having $\zeta = o(1)$ (which is the case when $\temp = o(1)$) rather than $0$ enables us to use the original version of this theorem almost automatically.

However, for the $\temp$ fixed case, we have $0< \zeta < 1$ fixed, yet $\zeta$ can be quite close to 1, and the mentioned theorem needs to be adjusted accordingly. Two definitions need to be modified: $$\mbox{$\Omega_0 = \{ \sigma \in \Omega : |S(\sigma) - \zeta|
\leq \frac{1-\zeta}{2}\}$}$$ and $$ \qquad\quad\mbox{$\Xi = \left\{ \sigma: \min\{ U(\sigma), U(\sigma_0) - U(\sigma), V(\sigma),
V(\sigma_0) - V(\sigma)  \} \geq \frac{(1-\zeta)^2 n}{20} \right\}$}~.$$
The remaining definitions and statement are all left without change, as well as the application of the theorem.
We further note that, with the above two modified definitions, following the same arguments of \cite{LLP} proves the required variant of the theorem.

\subsubsection*{Variance bound on the non-censored magnetization:}
In Lemma \ref{lem-var-bound-usual} we proved an upper bound of $O\left((\temp n)^{-1}\right)$
on $\var_{s_0} S_t$ for any $s_0 \geq \frac43{\sqrt\temp}$ and throughout a certain time interval.
In that proof, we used the fact that $\temp = o(1)$, giving a certain estimate on the required
time point $T_3$, which was then translated into a bound on the variance.

To prove the same statement for the case of $\temp$ fixed,
recall that Lemma \ref{lem-tau-3-bound}, discussed above, gives a bound of $1/n$ for
$\P_{s_0}(\tau_3 \leq t)$ and hence also for $\P_{s_0}(\tau_0 \leq t)$ (hitting $0$ rather than
$\zeta/3$). Plugging this into the proof of Lemma \ref{lem-var-bound-usual}, and using the fact
that $S_t$ is clearly bounded by 1, provides the required upper bound of $O(1/n)$ for the variance.

\section{Concluding remarks and open problems}\label{sec:conclusion}
In this work, we established cutoff for the censored Glauber dynamics on the mean-field Ising model.
It is widely believed that the behavior of the dynamics in the mean-field setting is essentially
the same as that for other underlying geometries, such as high dimensional tori. We therefore formulate
several conjectures following the insight that the mean-field model had recently provided.

Our results, together with those in the companion paper \cite{DLP}, reveal a symmetry around the critical
temperature, where the subcritical regime is analogous to the censored supercritical one. Namely,
the behavior below $\beta_c$ shows order $n^{3/2}$ mixing without cutoff at $\beta = 1-\temp$ for
$\temp = O(1/\sqrt{n})$, and cutoff with mixing order $(n/\temp)\log(\temp^2 n)$ whenever $\temp^2 n\to\infty$.
The same behavior was established for the censored dynamics above $\beta_c$,  only with a different
cutoff-constant in the case of $\temp^2 n \to\infty$.

In light of this, we have the following conjectures:
\begin{conjecture}
Consider the Glauber dynamics for the Ising model on a sequence of transitive graphs $\{G_n\}$. Then for a suitable notion of censoring
and any $|\temp| < \beta_c$,
there is cutoff for the original dynamics at $\beta_1 = \beta_c - \temp$ iff there is cutoff for the censored dynamics at $\beta_2 = \beta_c + \temp$.
\end{conjecture}
\begin{conjecture}
Consider the Glauber dynamics for the Ising model on a sequence of transitive graphs $\{G_n\}$. Then for a suitable notion of censoring
and any $|\temp| < \beta_c$,
the mixing-time $\tmix(\frac14)$ at $\beta_1 = \beta_c - \temp$ has precisely the same order as the mixing-time at $\beta_2 = \beta_c + \temp$.
\end{conjecture}

\end{document}